\documentclass[12pt, reqno]{amsart}

\usepackage{amsmath}
\allowdisplaybreaks
\usepackage{amsfonts}
\usepackage{amssymb}
\usepackage[all]{xy}           

\usepackage{bbding}
\usepackage{txfonts}
\usepackage{amscd}

\usepackage[shortlabels]{enumitem}
\usepackage{ifpdf}
\ifpdf
  \usepackage[colorlinks,final,backref=page,hyperindex]{hyperref}
\else
  \usepackage[colorlinks,final,backref=page,hyperindex,hypertex]{hyperref}
\fi
\usepackage{tikz}
\usepackage[active]{srcltx}

\makeatletter

\newtheorem{df}{Definition}[section]
\newtheorem{thm}{Theorem}[section]
\newtheorem{cor}{Corollary}[section]
\newtheorem{rem}{Remark}[section]

\newtheorem{prop}{Proposition}[section]
\newtheorem{exa}{Example}[section]
\newtheorem{lem}{Lemma}[section]

\numberwithin{equation}{section}

\begin{document}

\date{}
\title[Hom-Poisson superalgebras with nondegenerate bilinear forms and compatible Hom-anti-pre-Poisson superalgebras]
{Hom-Poisson superalgebras with nondegenerate bilinear forms and compatible Hom-anti-pre-Poisson superalgebras}

\author{W. Ben Abdelhafidh and O. Ncib }

\address{University of Sfax, Faculty of Sciences of Sfax,   BP 1171, 3000 Sfax, Tunisia}

\email{wiembenabdelhafidh@gmail.com}

\address{University of Gafsa, Faculty of Sciences Gafsa, 2112 Gafsa, Tunisia}

\email{othmenncib@yahoo.fr, othmen.ncib@fsgf.u-gafsa.tn}


\begin{abstract}

In this paper, we first introduce the notions of Hom-anti-Zinbiel superalgebras, extending the corresponding anti-algebraic structures to the Hom-super setting. We develop their fundamental properties and characterize them in terms of anti-super-$\mathcal O$-operators and anti-Rota-Baxter operators. Furthermore, we introduce the concepts of super-commutative Connes cocycles on super-commutative Hom-associative superalgebras and super-commutative $2$-cocycles on Hom-Lie superalgebras. We prove that nondegenerate super-commutative Connes cocycles give rise to compatible Hom-anti-Zinbiel superalgebra structures, while nondegenerate super-commutative $2$-cocycles induce compatible Hom-anti-pre-Lie superalgebra structures. As a main application, we show that a Hom-Poisson superalgebra equipped with a nondegenerate super-commutative Connes cocycle on its Hom-associative component and a nondegenerate super-commutative $2$-cocycle on its Hom-Lie component canonically determines a compatible Hom-anti-pre-Poisson superalgebra, and conversely. This provides a unified framework linking Hom-Poisson and Hom-anti-pre-Poisson superalgebras and reveals the fundamental role played by nondegenerate supersymmetric bilinear forms in the structure theory of Hom-type superalgebras. Several further structural results and characterizations are also obtained.
\end{abstract}
\date{\today}


\maketitle
\tableofcontents


\section{Introduction}

Pre-Lie algebras, also known as left-symmetric algebras, were first implicitly introduced by Cayley in 1890 \cite{Cayley} in his study of rooted tree algebras. Since then, they have naturally emerged in a wide variety of mathematical contexts, including the theory of convex homogeneous cones \cite{E.B. Vinberg}, affine manifolds and affine structures on Lie groups \cite{Koszul}, as well as the deformation theory of associative algebras \cite{Gerstenhaber}. Owing to their rich algebraic structure and broad applicability, pre-Lie algebras have become an important subject of research with deep connections to numerous areas of mathematics and mathematical physics. These include complex and symplectic structures on Lie groups and Lie algebras \cite{And-Sal,B.Y.Chu,Dardié-Medina1,Dardié-Medina2,Lichnerowicz-Medina}, phase spaces of Lie algebras \cite{Bai,B.A.Kupershmidt}, integrable systems \cite{M.Bordemann}, classical and quantum Yang--Baxter equations \cite{Diatta-Medina}, combinatorics \cite{Ebrahimi-Fard}, Poisson brackets, infinite-dimensional Lie algebras, vertex algebras, quantum field theory \cite{Connes-Kreimer}, and operad theory \cite{Chapoton-Livernet}.

Hom-pre-Lie algebras constitute a natural Hom-type generalization of pre-Lie algebras, in which the defining pre-Lie identity is twisted by a linear self-map, referred to as the structure map. This notion was introduced by Makhlouf and Silvestrov in \cite{Makhlouf-Silvestrov1}, providing a framework that extends classical pre-Lie theory while encompassing various deformations and generalized algebraic structures. Since their introduction, Hom-pre-Lie algebras have attracted considerable attention and have been investigated from several perspectives. Notable developments include their geometric aspects \cite{Zhang-Yu-Wang}, the theory of universal $\alpha$-central extensions \cite{Sun-Chen-Zhou}, and the study of Hom-pre-Lie bialgebras \cite{Sun-Li}.\\

Anti-pre-Lie algebras, introduced by G. Liu and C. Bai in \cite{Liu-Bai1}, arise as the anti-analogue of pre-Lie algebras and provide a natural counterpart to the latter within the framework of nonassociative algebras. Their superization was subsequently established by Zhao Chen, Shanshan Liu, and Liangyun Chen in \cite{Zhao-Liu-Chen}, leading to the notion of anti-pre-Lie superalgebras.


Motivated by the Hom-type deformation theory of algebraic structures, the notion of a Hom-anti-pre-Lie superalgebra was introduced in \cite{wi-oth}. A Hom-anti-pre-Lie superalgebra is a triple $(\mathcal{A},\circ,\alpha)$ consisting of a $\mathbb{Z}_2$-graded vector space $\mathcal{A}$, an even bilinear product $\circ$, and an even linear map $\alpha:\mathcal{A}\rightarrow\mathcal{A}$ satisfying, for all homogeneous elements $x,y,z\in\mathcal{A}$,
\begin{equation}\label{Hom-anti-pre-Lie1}
\alpha(x)\circ(y\circ z)-(-1)^{|x||y|}\alpha(y)\circ(x\circ z)
=(-1)^{|x||y|}[y,x]_\circ\circ\alpha(z),
\end{equation}
\begin{equation}\label{Hom-anti-pre-Lie2}
[x,y]_\circ\circ\alpha(z)
+(-1)^{|x|(|y|+|z|)}[y,z]_\circ\circ\alpha(x)
+(-1)^{|z|(|x|+|y|)}[z,x]_\circ\circ\alpha(y)
=0,
\end{equation}
where
\[
[x,y]_\circ=x\circ y-(-1)^{|x||y|}y\circ x.
\]
The identities \eqref{Hom-anti-pre-Lie1} and \eqref{Hom-anti-pre-Lie2} are called the anti-pre-Lie super-identities. Moreover, the supercommutator $[\cdot,\cdot]_\circ$ endows $\mathcal{A}$ with the structure of a Hom-Lie superalgebra. When $\alpha=\mathrm{Id}_{\mathcal A}$, one recovers the notion of an anti-pre-Lie superalgebra. Thus, Hom-anti-pre-Lie superalgebras provide a natural Hom-type generalization of anti-pre-Lie superalgebras and furnish a broader framework for studying deformations and twisted analogues of these algebraic structures.\\

Zinbiel algebras, also referred to as dual Leibniz algebras, were introduced by J.-L. Loday \cite{Loday} in the framework of Leibniz cohomology. They constitute the Koszul dual operad of Leibniz algebras, while the terminology \emph{Zinbiel}, coined by J.-M. Lemaire, is obtained by reversing the word \emph{Leibniz}, thereby emphasizing this dual relationship. The defining identity of a Zinbiel algebra can be viewed as a nonassociative decomposition of the associativity relation, endowing the corresponding operad with a rich combinatorial and algebraic structure. A fundamental property of Zinbiel algebras is that the symmetrization of the Zinbiel product always defines a commutative associative algebra, establishing a profound connection between Zinbiel and associative algebraic structures. Owing to these remarkable features, Zinbiel algebras have found significant applications in Leibniz cohomology, operad theory, shuffle and divided power algebras, deformation theory, multiple zeta values, Rota-Baxter algebras, and related areas of algebra and mathematical physics \cite{Loday,Ginzburg-Kapranov,Dokas,Guo-Zhang,Zinbiel-bialgebra}. Motivated by the intrinsic duality between pre-Lie-type and Zinbiel-type structures, Chengming Bai and collaborators introduced \emph{anti-Zinbiel algebras} as the anti-counterpart of Zinbiel algebras (see, for instance, \cite{Bai-AntiZinbiel}). In contrast to Zinbiel algebras, whose symmetrized multiplication is commutative associative, anti-Zinbiel algebras are intimately connected with anti-pre-Lie algebras and provide a natural class of nonassociative algebras exhibiting rich algebraic and geometric features. Since their introduction, anti-Zinbiel algebras have attracted considerable attention due to their close relationships with anti-pre-Lie algebras, $\mathcal{O}$-operators, matched pairs, bialgebra structures, generalized Yang-Baxter equations, and various deformation and representation theories. Following the emergence of Hom-type algebras initiated by Hartwig, Larsson, and Silvestrov and further developed by Makhlouf and Silvestrov, Hom-type analogues of Zinbiel and anti-Zinbiel algebras were obtained by twisting their defining identities via a linear self-map. These Hom-generalizations not only recover the classical structures when the twisting map is the identity but also substantially enlarge the scope of the theory by incorporating deformation phenomena and twisted symmetries. In particular, Hom-Zinbiel and Hom-anti-Zinbiel algebras furnish an appropriate framework for investigating representations, matched pairs, bialgebras, Manin triples, $\mathcal{O}$-operators, and related algebraic structures. Consequently, they occupy an increasingly important position in the theory of Hom-algebras, revealing deep structural connections with Hom-associative, Hom-Leibniz, Hom-pre-Lie, Hom-Poisson, and other Hom-type algebraic systems. The primary objective of the present paper is to extend the theory of anti-Zinbiel algebras to the Hom-super setting by introducing the notion of Hom-anti-Zinbiel superalgebras. This new structure simultaneously generalizes anti-Zinbiel algebras, Hom-anti-Zinbiel algebras, and anti-Zinbiel superalgebras, thereby providing a unified framework in which the Hom-twisting and $\mathbb{Z}_2$-grading coexist. We investigate the fundamental properties of Hom-anti-Zinbiel superalgebras and establish their basic structural theory.\\

Poisson algebras constitute one of the fundamental algebraic structures arising in differential geometry, Hamiltonian mechanics, deformation quantization, and mathematical physics. A Poisson algebra is simultaneously a commutative associative algebra and a Lie algebra whose two structures are related through the Leibniz identity, expressing the compatibility between the associative multiplication and the Lie bracket. Since their appearance in the study of classical mechanics, Poisson algebras have become indispensable tools in symplectic geometry, integrable systems, and quantization theory. The algebraic theory of Poisson algebras has been extensively developed and is now a central topic in both pure and applied mathematics; see, for instance, \cite{Kosmann-Schwarzbach,Lichnerowicz}.

Motivated by the theory of dendriform algebras introduced by Loday, Aguiar introduced the notion of pre-Poisson algebras, which provide a splitting of the Poisson algebraic structure in much the same way that dendriform algebras split associative algebras and pre-Lie algebras split Lie algebras \cite{Aguiar}. A pre-Poisson algebra consists of a permutative algebra together with a pre-Lie algebra satisfying suitable compatibility conditions. The commutator of the pre-Lie product and the symmetrization of the permutative multiplication naturally recover a Poisson algebra. Consequently, pre-Poisson algebras furnish an algebraic framework underlying Poisson geometry, classical $r$-matrices, deformation theory, and various splitting constructions in operad theory. Their structural properties, cohomology, and deformation theory have subsequently attracted considerable attention.

More recently, inspired by the theory of anti-pre-Lie algebras and anti-dendriform-type structures, Bai and collaborators introduced the notion of anti-pre-Poisson algebras, which may be regarded as an anti-version of pre-Poisson algebras \cite{Liu-Bai1}. In contrast to the pre-Poisson case, the associative component is replaced by an anti-permutative algebra while the Lie-theoretic component is governed by an anti-pre-Lie algebra, together with a collection of compatibility identities reflecting the anti-symmetric nature of the underlying operations. Anti-pre-Poisson algebras naturally generalize anti-pre-Lie algebras and provide an appropriate algebraic framework for studying noncommutative Poisson-type structures. Their theory has revealed close connections with $\mathcal O$-operators, Connes cocycles, matched pairs, bialgebra structures, and Manin triples, thereby extending many classical constructions from Poisson geometry to the anti-setting.\\

\textbf{Commutative $2$-cocycles} on Lie algebras and \textbf{commutative Connes cocycles} on commutative associative algebras have recently attracted considerable attention as symmetric analogues of symplectic and Frobenius structures, respectively. Let $(\mathcal A,[\cdot,\cdot])$ be a Lie algebra. A commutative $2$-cocycle on $\mathcal A$ is a symmetric bilinear form satisfying
$$
\mathfrak B([x,y],z)+\mathfrak B([z,x],y)+\mathfrak B([y,z],x)=0,
\;\forall x,y,z\in\mathcal A.
$$
This notion was systematically investigated by Dzhumadil'daev and Zusmanovich \cite{DzhZu}, who established its fundamental properties and its close relationship with antiderivations. Subsequently, Liu and Bai \cite{Liu-Bai1} introduced anti-pre-Lie algebras and proved that every Lie algebra endowed with a nondegenerate commutative $2$-cocycle carries a canonical anti-pre-Lie algebra structure whose sub-adjacent Lie algebra is precisely the original Lie algebra. Conversely, every anti-pre-Lie algebra naturally determines a nondegenerate commutative $2$-cocycle on its associated Lie algebra. This correspondence provides the symmetric counterpart of the classical relationship between symplectic Lie algebras and pre-Lie algebras and has further connections with anti-$\mathcal O$-operators, admissible Novikov algebras, anti-pre-Lie-Poisson algebras and related bialgebraic structures \cite{Liu-Bai1,LiuBaiAPB}.

Dually, let $(\mathcal A,\mu)$ be a commutative associative algebra. A commutative Connes cocycle on $\mathcal A$ is a symmetric bilinear form satisfying
$$
\mathfrak B(\mu(x,y),z)+
\mathfrak B(\mu(y,z),x)+
\mathfrak B(\mu(z,x),y)=0,
\;\forall x,y,z\in\mathcal A.
$$
This notion may be viewed as the commutative counterpart of the classical Connes cocycle introduced in cyclic cohomology. Building upon the theory of Connes cocycles and dendriform algebras developed by Bai \cite{BaiConnes}, Gao, Liu and Bai established that nondegenerate commutative Connes cocycles are in one-to-one correspondence with anti-Zinbiel algebras (equivalently, commutative anti-dendriform algebras). More precisely, every commutative associative algebra equipped with a nondegenerate commutative Connes cocycle admits a compatible anti-Zinbiel algebra structure, while every anti-Zinbiel algebra induces a commutative associative algebra together with a canonical commutative Connes cocycle. Consequently, the pair
(\text{commutative associative algebra},\ \text{commutative Connes cocycle})
plays a role completely analogous to that of
(\text{Lie algebra},\ \text{commutative } $2$\text{-cocycle}),
forming the associative counterpart of the correspondence between Lie algebras and anti-pre-Lie algebras \cite{Dongfang-Liu-Bai}.

Motivated by the remarkable development of Hom-type algebras and superalgebras, it is natural to investigate whether these correspondences remain valid in the Hom-super framework. In this paper, we introduce the notions of super-commutative $2$-cocycles on Hom-Lie superalgebras and super-commutative Connes cocycles on super-commutative Hom-associative superalgebras. The former are supersymmetric even bilinear forms satisfying the Hom-twisted cocycle identity associated with the Hom-Lie superbracket, whereas the latter are supersymmetric even bilinear forms satisfying the corresponding Hom-twisted Connes cocycle identity with respect to the Hom-associative multiplication. We prove that nondegenerate super-commutative $2$-cocycles characterize Hom-anti-pre-Lie superalgebras, while nondegenerate super-commutative Connes cocycles characterize Hom-anti-Zinbiel superalgebras. These results extend the classical correspondences between Lie algebras and anti-pre-Lie algebras, and between commutative associative algebras and anti-Zinbiel algebras, to the Hom-super setting, thereby providing a unified framework for studying compatible supersymmetric bilinear forms on Hom-type superalgebras. Furthermore, as an application of the above characterizations, we establish a correspondence between Hom-Poisson superalgebras and their compatible Hom-anti-pre-Poisson superalgebras through the framework of nondegenerate supersymmetric bilinear forms. More precisely, we prove that every Hom-Poisson superalgebra equipped with a nondegenerate super-commutative Connes cocycle on its super-commutative Hom-associative part and a nondegenerate super-commutative $2$-cocycle on its Hom-Lie superalgebra part canonically induces a compatible Hom-anti-pre-Poisson superalgebra structure, and conversely. This correspondence generalizes the well-known relationships between Poisson algebras and compatible anti-pre-Poisson algebras to the Hom-super setting, thereby providing a conceptual bridge between Hom-Poisson geometry and Hom-anti-pre-Poisson structures through invariant supersymmetric bilinear forms.
\subsection*{Organization} The paper is organized as follows:\\

In Section \ref{Sec2}, we recall the basic notions and preliminary results concerning super-commutative Hom-associative superalgebras, Hom-Lie superalgebras, their representations and dual representations, as well as Hom-Zinbiel superalgebras. We also review quadratic structures and introduce super-commutative Connes cocycles on super-commutative Hom-associative superalgebras, providing the foundational material required for the subsequent developments. In Section \ref{Sec3}, we introduce the notions of Hom-anti-Zinbiel superalgebras and Hom-anti-pre-Lie superalgebras. Furthermore, we define super-commutative Connes cocycles and super-commutative $2$-cocycles in the Hom-super setting and establish their relationships with Hom-anti-Zinbiel and Hom-anti-pre-Lie superalgebras, respectively. Several characterizations in terms of anti-super-$O$-operators and anti-Rota-Baxter operators are also obtained. In Section \ref{Sec4}, we investigate Hom-Poisson superalgebras endowed with nondegenerate supersymmetric bilinear forms. By combining the correspondence between super-commutative Connes cocycles and Hom-anti-Zinbiel superalgebras with the correspondence between super-commutative $2$-cocycles and Hom-anti-pre-Lie superalgebras, we establish a one-to-one relationship between Hom-Poisson superalgebras and their compatible Hom-anti-pre-Poisson superalgebras. This provides a unified framework linking Hom-Poisson structures and Hom-anti-pre-Poisson structures through nondegenerate supersymmetric bilinear forms.

Throughout this paper, unless otherwise specified, all vector spaces are assumed to be
finite-dimensional over a field $\mathbb{K}$ of characteristic $0$.

\section{Preliminaries and basic results}\label{Sec2} 

In this section, we recall the basic definitions and fundamental results that will be used throughout the paper. In particular, we review super-commutative Hom-associative and Hom-Lie superalgebras, their representations, quadratic structures, and related notions. These preliminaries provide the algebraic framework for the construction and study of quadratic Hom-Poisson superalgebras and the main results established in the subsequent sections.

\subsection{Hom-associative and Hom-Zinbiel superalgebras}

\begin{df}(\cite{Faouzi-Abdenacer})
A \textbf{super-commutative Hom-associative superalgebra} is a triplet $(\mathcal{A},\mu,\alpha)$ consisting of a $\mathbb{Z}_2$-graded vector space $\mathcal{A}$, an even linear map
 $\alpha:\mathcal{A}\rightarrow\mathcal{A}$, (i.e: $\alpha(\mathcal{A}_i)\subseteq\mathcal{A}_i)$ and an even bilinear map
 $\mu:\mathcal A\times \mathcal{A}\rightarrow\mathcal{A}$, (i.e: $\mu(\mathcal{A}_i,\mathcal{A}_j)\subseteq\mathcal{A}_{i+j})$ such that $\mu$ is supercommuatative, that is $\mu(x,y)=(-1)^{|x||y|}\mu(y,x)$ and the following condition hold:
 $$ass_{\mu}^\alpha(x,y,z)=0\;(\text{Hom-associativity}),$$
 where $$ass_{\mu}^\alpha(x,y,z)=\mu(\mu(x,y),\alpha(z))-\mu(\alpha(x),\mu(y,z)),\;\forall x,y,z\in\mathcal H(\mathcal{A}).$$
\end{df}
If in addition $\mu(\alpha(x),\alpha(y))=\alpha(\mu(x,y)),\;\forall\;x,y\in\mathcal{H}(\mathcal{A})$),
the Hom-associative superalgebra $(\mathcal{A},\mu,\alpha)$ is said to be multiplicative.\\

\begin{df}(\cite{O-N})
 A \textbf{representation} of a super-commutative Hom-associative superalgebra $(\mathcal A, \mu, \alpha)$  is a quadruple $(V,\eta,\beta)$ where $V$ is a $\mathbb Z_2$-graded vector space, $\beta\in gl(V)$ and $\eta: \mathcal A\rightarrow gl(V)$ are two even linear maps
such that the following condition hold for all $x, y \in\mathcal A:$
\begin{equation}\label{rAs2}
\eta(\mu(x,y))\beta=\eta(\alpha(x))\eta(y)
\end{equation}
\end{df}
\begin{rem}
\begin{enumerate}
\item 
Since $\mu$ is super-commutative, Eq. \eqref{rAs2} also implies
\begin{equation}\label{rAs3}
\eta(\alpha(x))\eta(y)=(-1)^{|x||y|}\eta(\alpha(y))\eta(x). \end{equation}
\item Assume that $(\mathcal A,\mu,\alpha)$ is a multiplicative Hom-associative superalgebra. Then the representation map $\eta$ is compatible with the twisting maps $\alpha$ and $\beta$, namely,
\begin{equation}\label{rAs1}
\eta(\alpha(x))\beta=\beta\eta(x),\;\forall x\in \mathcal H(\mathcal A).    
\end{equation}
\end{enumerate}
\end{rem}
In particular, $(\mathcal A, \mathfrak L, \alpha)$ is a representation over $(\mathcal A,\mu,\alpha)$ which is called the adjoint representation, where $\mathfrak L: \mathcal A\to End_\mathbb K(\mathcal A)$ is an even linear map defined by 
\begin{equation}\label{adj-rep-H-ass-sup}
\mathfrak L(x)(y)=\mu(x,y),\; \forall x,y\in\mathcal A.
\end{equation}
Suppose that $(\mathcal{A},\mu,\alpha)$ is a super-commutative Hom-associative superalgebra.
Let $V$ be a $\mathbb{Z}_2$-graded vector space, $\eta:\mathcal A\to gl(V)$ and $\beta\in gl(V)$ are two even linear maps. Then $(V,\eta,\beta)$ is a representation of $(\mathcal{A},\mu,\alpha)$ if and only if $(\mathcal{A}\oplus V,\mu_{\mathcal{A}\oplus V},\alpha\oplus\beta)$
is a super-commutative Hom-associative superalgebra, where $\mu_{\mathcal{A}\oplus V}$ and $(\alpha+\beta)$
are defined for all $x,y\in\mathcal{H}(\mathcal{A}),\;u,v \in\mathcal{H}(V)$ by
\begin{eqnarray}
    \mu_{\mathcal{A}\oplus V}(x+u,y+v)&=&\mu(x,y)+\eta(x)v+(-1)^{|y||u|}\eta(y)u,\label{Hom-ass-direct-sum1}\\
    (\alpha\oplus\beta)(x+u)&=&\alpha(x)+\beta(u)\label{Hom-ass-direct-sum2}.
\end{eqnarray}
This Hom-associative superalgebra is called semi-direct product of $(\mathcal{A},\mu,\alpha)$ and $(V,\eta,\beta)$ and denoted by
 $\mathcal{A}\ltimes^{\alpha}_{\eta,\beta}V$ or simply $\mathcal{A}\ltimes_{\eta} V$.\\

Let $\mathcal A$ and $V$ be two $\mathbb Z_2$-graded vector spaces. For even linear map $\eta:\mathcal A\to gl(V)$, we set $\eta^\ast:\mathcal A\to V^\ast$ by

\begin{equation}\label{n-dual-rep-H-ass-sup}
    \langle\eta^*(a)(\xi),u\rangle=-\langle\xi,\eta(a)(u)\rangle.
\end{equation}

Here $\langle\cdot,\cdot\rangle$, is the usual pairing between $V$ and $V^\ast$.
If $(V,\eta,\beta)$ is a representation of a  super-commutative multiplicative Hom-associative superalgebra $(\mathcal A,\mu,\alpha)$, then $(\eta^\ast)$ does not generally define a representation of $\mathcal A$ anymore.

Let us define $\eta^\star:A\longrightarrow gl(V^*)$ by
\begin{eqnarray}
  \label{eq:1.3}\eta^\star(a)(\xi)&:=&\eta^*(\alpha(a))\big{(}(\beta^{-2})^*(\xi)\big{)},\;\forall a\in\mathcal H(\mathcal A),\;\xi\in\mathcal H(V^\ast).
\end{eqnarray}

\begin{thm}\label{dual-rep-H-ass-sup}
Let $(V,\eta,\beta)$ be a representation of a super-commutative multiplicative Hom-associative superalgebra $(\mathcal A,\mu,\alpha)$. Then $(V^*,-\eta^\star,(\beta^{-1})^*)$ is a representation of $(\mathcal A,\mu,\alpha)$, which is called the dual representation of $(V,\eta,\beta)$.
\end{thm}
In particular $(\mathcal A^\ast,-\mathfrak L^\star,(\alpha^{-1})^\ast)$ is a representation of $(\mathcal A,\mu,\alpha)$ which is called the coadjoint representation of $(\mathcal A,\mathfrak L,\alpha)$ where $\mathfrak L$ and $\mathfrak R$ are defined by \eqref{adj-rep-H-ass-sup}.\\

Let $(\mathcal A,\mu,\alpha)$ be a super-commutative Hom-associative  superalgebra and $(V,\eta,\beta)$ be a representation. Recall that an even linear map $T:V\to\mathcal A$ is called a \textbf{super-$\mathcal O$-operator} on $(\mathcal A,\mu,\alpha)$ associate with $(V,\eta,\beta)$ if the following equations holds:
\begin{eqnarray}
T\circ\beta&=&\alpha\circ T,\label{cond-O-oper-ncomm-Hom-ass1} \\
\mu( T(u),T(v))&=&T\big(\eta(T(u))v+(-1)^{|u||v|}\eta(T(v))u\big),\;\forall u,v\in\mathcal H(V).\label{cond-O-oper-ncomm-Hom-ass2}
\end{eqnarray}
In particular, a super-$\mathcal O$-operator $\mathcal R$ of $(\mathcal A,\mu,\alpha)$ associated with the representation $(\mathcal A,\mathfrak L,\alpha)$
is called an \textbf{anti-Rota–Baxter operator (of weight zero)}; that is, $\mathcal R :\mathcal A\to\mathcal A $ is an even linear map commuting with $\alpha$ and satisfying
\begin{equation}\label{RB-oper-noncomm-ass}
 \mu(\mathcal R(x),\mathcal R(y))=\mathcal R\big(\mu(\mathcal R(x),y)+\mu(x,\mathcal R(y))\big),\;\forall x,y\in\mathcal H(\mathcal A).   
\end{equation}

\begin{thm}\label{Hom-Zinb-by-O-oper}
Let $T:V\rightarrow\mathcal A$ be a super-$\mathcal O$-operator on a super-commutative Hom-associative superalgebra $(\mathcal A,\mu,\alpha)$ with respect to a representation $(V,\eta,\beta)$. Let us define a binary operation $\star_T:V\otimes V\to V$ by:
\begin{equation}\label{H-Zinb-from-O-oper-H-ass}  
u \star_T v=\eta(T(u))v,\;\forall u,v\in\mathcal H(V).
\end{equation}
Then $(V,\star_T,\beta)$ is a Hom-zinbiel superalgerbra. In this case $T$ is a homomorphism of super-commutative Hom-associative superalgebras from $(V,\mu_{\star_T},\beta)$ to $(\mathcal A,\mu,\alpha)$. Furthermore, there is an induced Hom-Zinbiel superalgebra structure on
$T(V)=\{T(u),\;u\in V\}\subseteq A$ given by
\begin{equation}\label{Hom-Zinb-T(V)}
T(u)\star T(v)=T(u\star_Tv),\;\forall u,v\in\mathcal H(V).
\end{equation}
\end{thm}
\begin{proof}
Let $u,v,w \in \mathcal H(V)$  
\begin{align*}
(u\star_T v) \star_T \beta(w)+(-1)^{|u||v|}(v\star_T u)\star_T \beta(w)&=(\eta(T(u))v)\star_T \beta(w)+(-1)^{|u||v|} (\eta(T(v))u)\star_T \beta(w) \\  
&=\eta(T(\eta(T(u))v+(-1)^{|u||v|} \eta (T(V))u)\beta(w) \\
&\stackrel{\eqref{cond-O-oper-ncomm-Hom-ass2}}{=}\eta(\mu(T(u), T(v)))\beta(w)
\end{align*}
\begin{align*}
\beta(u)\star_T (v\star_T w)&=\beta(u)\star_T \eta(T(v))w\\
&=\eta(\alpha T((u)))\eta(T(v))w
\end{align*}
Since $(V,\eta,\beta)$ is a representation of $(\mathcal A,\mu,\alpha)$, we have
\begin{align*}
\eta(\mu(T(u), T(v)))\beta(w)=\eta(\alpha T((u)))\eta(T(v))w
\end{align*}
Then,$(V,\star_T,\beta)$ is a Hom-zinbiel superalgerbra.
\end{proof}
\begin{cor}\label{H-Zinb-to-Hom-ass-RB}
Let $\mathcal R$ be a Rota-Baxter operator of weight zero on a supercommutative Hom-associative superalgebra  $(\mathcal A,\mu,\alpha)$. Then, $(\mathcal A,\star_\mathcal R,\alpha)$ where $\star_\mathcal R:\mathcal A\to\mathcal A$ is defined by 
\begin{equation}\label{H-anti-dend-From-ARB-Hass} 
 x\star_\mathcal R y=\mu(\mathcal R(x),y),
\end{equation}
for any $x,y\in\mathcal H(\mathcal A)$ is a Hom-Zinbiel superalgebra.
\end{cor}

We recall some facts on vector superspaces from \cite{Bai-Guo-Zhang, Cheng-Wang, M-Scheunert}. Let  $V=V_{\bar 0}\oplus
V_{\bar 1}$ and $W=W_{\bar 0}\oplus W_{\bar 1}$ be two vector superspaces over a field $\mathbb K$.
A bilinear form $\mathfrak B: V\times W \longrightarrow \mathbb K$
is called
 odd  if $\mathfrak B(V_{\bar{0}}, W_{\bar{0}})=\mathfrak B(V_{\bar{1}}, W_{\bar{1}})=0$, and
 even  if $\mathfrak B(V_{\bar{0}}, W_{\bar{1}})=\mathfrak B(V_{\bar{1}}, W_{\bar{0}})=0$;
nondegenerate if  $\mathfrak B(v, w)=0$
for all $w\in W$ implies $v=0$, and $\mathfrak B(v, w)=0$
for all $v\in V$ implies $w=0$.
The linear dual $V^*={\rm Hom}(V, \mathbb  K)$  of $V$ inherits a
$\mathbb{Z}_2$-graduation $V^*=V^*_{\bar
0}\oplus V^*_{\bar 1}$ with
\begin{equation}\label{eq:2.3}
    V^*_\alpha:=\big\{u^*\in V^*| u^*(V_{\alpha+\bar
        1})=\{0\}\big\},\;\; \forall \alpha\in \mathbb{Z}_2.
\end{equation}

\begin{df}
Let $(\mathcal A,\mu,\alpha)$ be a super-commutative Hom-associative superalgebra. An even bilinear form $\mathfrak{B}:\mathcal A\times\mathcal A\to \mathbb K$ is called:
\begin{enumerate}
\item[(1)] \textbf{supersymmetric} if $\mathfrak{B}(x,y)=(-1)^{|x||y|}\mathfrak{B}(y,x),\,\,\ \forall x,y\in\mathcal{H}(\mathcal A)$
\item[(2)] \textbf{super-skewsymmetric} if $\mathfrak{B}(x,y)=-(-1)^{|x||y|}\mathfrak{B}(y,x),\;\forall x,y\in\mathcal{H}(\mathcal A)$.
\item[(3)] \textbf{Invariant} if 
\begin{equation}\label{inv-bil-form-H-ass-sup}
 \mathfrak B(\mu(x,y),\alpha(z))=\mathfrak B(\alpha(x),\mu(y,z)),\;\forall x,y,z\in\mathcal H(\mathcal A).
\end{equation}
\end{enumerate}
\end{df}
Suppose that $(\mathcal A,\mu,\alpha)$ is a super-commutative Hom-associative superalgebra. Then the natural nondegenerate symmetric
bilinear form $\mathfrak B_d$ on $\mathcal A\oplus\mathcal A^\ast$ defined by 
\begin{equation}\label{natur-biln-form-H-assoc-sup}
 \mathfrak B_d((x,a^\ast),(y,b^\ast))=<x,b^\ast>+<a^\ast,y>,\;\forall x,y\in\mathcal H(\mathcal A),\;a^\ast,b^\ast\in\mathcal H(\mathcal A^\ast),   
\end{equation}
is invariant on the super-commutative Hom-associative superalgebra $\mathcal A\ltimes_{-\mathfrak L^\star_\mu}\mathcal A^\ast$.
\begin{df}\label{Hom-dend-superalg}
 \textbf{(Hom-Zinbiel superalgebra).} Let $\mathcal A$ be a super vector space together with an even bilinear operation
\[
\star:\mathcal A\times \mathcal A\to \mathcal A.
\]
and an even linear map \[
\alpha:\mathcal A\to\mathcal A
\]
$(\mathcal A,\star,\alpha)$ is called a \textbf{Hom-Zinbiel superalgebra} if
\begin{equation}\label{cond-H-zinb-sup}
\alpha(x)\star (y\star z)
=
(x\star y)\star \alpha(z)
+(-1)^{|x||y|}
(y\star x)\star \alpha(z),
\qquad \forall x,y,z\in\mathcal H(\mathcal A).
\end{equation}
\end{df}
\begin{prop}\label{Hom-Zinb-to-Hom-ss}
Let $\mathcal A$ be a $\mathbb Z_2$-graded vector space equipped with an even bilinear map $\star:\mathcal A\times\mathcal A\to\mathcal A$. Define a bilinear map $\mu_\star:\mathcal A\times\mathcal A \to\mathcal A$ by
\begin{equation}\label{comp-comm-hom-ass-of-Hom-Zinb}
\mu_\star(x,y)=x\star y+(-1)^{|x||y|}y\star x,\;\forall x,y\in\mathcal H(\mathcal A). 
\end{equation}
Then the following conditions are equivalent:  
\begin{enumerate}
\item $(\mathcal A,\star,\alpha)$ is a Hom-Zinbiel superalgebra.
\item $(\mathcal A,\mu,\alpha)$ is a super-commutative Hom-associative superalgebra and $(\mathcal A,\mathfrak L_\star,\alpha)$ is a representation of  $(\mathcal A,\mu,\alpha)$.
\item There is a super-commutative Hom-associative superalgebra structure on $\mathcal A\oplus\mathcal A$ defined by
\begin{eqnarray}
(\alpha\oplus\alpha)(x,a)&=&(\alpha(x),\alpha(a)),\label{struc-supcomm-H-ass1}\\
\mu_{\mathcal A\oplus\mathcal A}((x,a),(y,b))&=&(x\star y+(-1)^{|x||y|}y\star x,x\star b+(-1)^{|a||y|}y\star a),\label{struc-supcomm-H-ass2}
\end{eqnarray}
for all $x,y,a,b\in\mathcal H(\mathcal A)$.
\end{enumerate}
$(\mathcal A,\mu,\alpha)$ is called the \textbf{sub-adjacent super-commutative Hom-associative superalgebra} of $(\mathcal A,\star,\alpha)$ and $(\mathcal A,\star,\alpha)$ is called a \textbf{compatible Hom-Zinbiel superalgebra} of $(\mathcal A,\mu,\alpha)$.
\end{prop}
\begin{proof}
Straightforward.    
\end{proof}
\begin{df}
An even super-skew-symmetric bilinear form $\mathfrak B$ on a super-commutative Hom-associative superalgebra $(\mathcal A,\mu,\alpha)$ is
called a \textbf{Connes cocycle} if it satisfies
\begin{equation}\label{def-con-cocy-supcomm-H-ass-sup}
 \mathfrak B(\mu(x,y),\alpha(z))+(-1)^{|x|(|y|+|z|)}\mathfrak B(\mu(y,z),\alpha(x))+(-1)^{|z|(|x|+|y|)}\mathfrak B(\mu(z,x),\alpha(y))=0,     
 \end{equation}
 for any homogeneous $x,y,z\in\mathcal A$.
\end{df}
\begin{prop}\label{comp-H-ant-Zinb-from-Connes-Coc}
Let $\mathfrak B$ be a nondegenerate Connes cocycle on a multiplicative super-commutative
Hom-associative superalgebra $(\mathcal A,\mu,\alpha)$. Then there is a compatible
Hom-Zinbiel superalgebra $(\mathcal A,\star,\alpha)$ of $(\mathcal A,\mu,\alpha)$ defined by
\begin{equation}\label{h-Zinb-via-Connes-coc}
\mathfrak B(\alpha(x)\star\alpha(y),\alpha^2(z))=(-1)^{|x||y|}\mathfrak B(\alpha(y),\mu(x,z),\qquad \forall\,x,y,z\in\mathcal H(\mathcal A).
\end{equation}

Conversely, let $(\mathcal A,\star,\alpha)$ be a Hom-Zinbiel superalgebra and
$(\mathcal A,\mu,\alpha)$ be the sub-adjacent super-commutative Hom-associative superalgebra.
Then the natural nondegenerate super-skew-symmetric bilinear form
$\mathfrak B_{d}$ defined by Eq. \eqref{natur-biln-form-H-assoc-sup} is a Connes cocycle on the super-commutative Hom-associative superalgebra
$\mathcal A\ltimes_{-\mathfrak L_{\mu}^{\star}}\mathcal A^{*}$.    
\end{prop}
\begin{proof}
Define a linear map $T: \mathcal A\rightarrow \mathcal A^{*}$ by \begin{equation}\label{eq:Hom-ass-T} \langle T(x),y\rangle=\mathfrak{B}(x,y),\qquad \forall x,y\in \mathcal H(\mathcal A). \end{equation} Since $\mathfrak B$ is nondegenerate, the map $T$ is invertible. Let $a^{*},b^{*}\in \mathcal A^{*}$. Then there exist $x,y\in\mathcal H(\mathcal A)$ such that $a^{*}=T(x)$ and $b^{*}=T(y)$. For any $z\in\mathcal H(\mathcal A)$, we compute  $$\mathfrak B(\mu(x,y),\alpha(z)) =\langle T(\mu(x,y)),\alpha(z)\rangle =\langle T(\mu(T^{-1}(a^{*}),T^{-1}(b^{*}))),\alpha(z)\rangle.$$
By Eq.\eqref{dual-adjoint-rep}, we have
\begin{eqnarray*}
\mathfrak B(\mu(y,z),\alpha(x)) &=&-(-1)^{|x|(|y|+|z|)}\langle T(\alpha(x)),\mu(y,z)\rangle =(-1)^{|x||z|}\langle \mathfrak{L}^{\star}(\alpha(y))T(\alpha(x)),\alpha^2(z)\rangle\\
&=&(-1)^{|x||z|}\langle \mathfrak{L}^{\star}(\alpha(y))(\alpha^{-1})^*T(x),\alpha^2(z)\rangle\\
&=&(-1)^{|x||z|}\langle (\alpha^{-1})^*\mathfrak{L}^{\star}(y)T(x),\alpha^2(z)\rangle\\
&=&(-1)^{|x||z|}\langle (\mathfrak{L}^{\star}(y)T(x),\alpha(z)\rangle\\
&=&(-1)^{|a||z|}\langle \mathfrak{L}^{\star}(T^{-1}(b^{*}))a^{*},\alpha(z)\rangle.
\end{eqnarray*}
Similarly we have
$$\mathfrak B(\mu(z,x),\alpha(y))=(-1)^{|x||z|}\mathfrak B(\mu(x,z),\alpha(y))=(-1)^{|z|(|a|+|b|)}\langle \mathfrak{L}^{\star}(T^{-1}(a^{*}))b^{*},\alpha(z)\rangle.$$ 
   
Since $\mathfrak B$ is a Connes cocycle on the Hom-associative superalgebra $(\mathcal A,\mu,\alpha)$, we have \[ \mathfrak B(\mu(x,y),\alpha(z)) +(-1)^{|x|(|y|+|z|)}\mathfrak B(\mu(y,z),\alpha(x)) +(-1)^{|z|(|x|+|y|)}\mathfrak B(\mu(z,x),\alpha(y))=0 . \] Therefore \[ \mu(T^{-1}(a^{*}),T^{-1}(b^{*})) = T^{-1}\big(-\mathfrak{L}^{\star}(T^{-1}(a^{*}))(b^{*}) -(-1)^{|a||b|} \mathfrak{L}^{\star}(T^{-1}(b^{*}))(a^{*})  \big). \] Thus $T^{-1}$ is an invertible $\mathcal O$-operator of the Hom-associative superalgebra $(\mathcal A,\mu,\alpha)$ associated with the representation $(\mathcal A^{*},-\mathfrak{L}^{\star},(\alpha^{-1})^{*})$. Then, by using Corollary \ref{H-Zinb-to-Hom-ass-RB}, there exists a compatible Hom-Zinbiel superalgebra structure $\star$ on $(\mathcal A,\mu,\alpha)$ defined by \begin{equation}\label{eq:Hom-assoc-product} x\star y=T^{-1}\big(-\mathfrak{L}^{\star}(x)T(y)\big),\qquad \forall x,y\in\mathcal H(\mathcal A). \end{equation} Finally, for any $x,y,z\in\mathcal H(\mathcal A)$, we obtain 
    \begin{eqnarray*}
     \mathfrak B(\alpha(x)\star\alpha( y),\alpha^2(z)) &=&\langle T(\alpha(x)\star\alpha( y)),\alpha^2(z)\rangle\\&=&-\langle \mathfrak{L}^{\star}(\alpha(x))T(\alpha(y)),\alpha^2(z)\rangle\\&=&(-1)^{|x||y|}\langle T(\alpha(y)),\mu(x,z)\rangle \\&=&(-1)^{|x||y|}\mathfrak B(\alpha(y),\mu(x,z)). 
    \end{eqnarray*}
    Hence the Hom-Zinbiel product $\star$ satisfies \[ \mathfrak B(\alpha(x)\star\alpha( y),\alpha^2(z))=(-1)^{|x||y|}\mathfrak B(\alpha(y),\mu(x,z)),\; \forall x,y,z\in\mathcal H(\mathcal A). \] Which gives the result. The converse follows by a direct computation.
\end{proof}
\begin{cor}\label{Cor-comp-mult-H-ant-Zinb-from-Connes-Coc}
 Let $\mathfrak B$ be a nondegenerate Connes cocycle on a multiplicative super-commutative
Hom-associative superalgebra $(\mathcal A,\mu,\alpha)$ satisfying
\begin{equation}\label{cond-mult-bil-form}
\mathfrak B(\alpha(x),\alpha(y))=\mathfrak B(x,y),\;\forall x,y\in\mathcal H(\mathcal A).  
\end{equation}
 Then there is a compatible
Hom-Zinbiel superalgebra $(\mathcal A,\star,\alpha)$ of $(\mathcal A,\mu,\alpha)$ defined by
\begin{equation}\label{h-Zinb-via-Connes-coc-mult-Hom-assoc}
\mathfrak B(x\star y,\alpha(z))=(-1)^{|x||y|}\mathfrak B(\alpha(y),\mu(x,z),\qquad \forall\,x,y,z\in\mathcal H(\mathcal A).
\end{equation}

Conversely, let $(\mathcal A,\star,\alpha)$ be a Hom-Zinbiel superalgebra and
$(\mathcal A,\mu,\alpha)$ be the sub-adjacent super-commutative Hom-associative superalgebra.
Then the natural nondegenerate super-skew-symmetric bilinear form
$\mathfrak B_{d}$ defined by Eq. \eqref{natur-biln-form-H-assoc-sup} is a Connes cocycle on the commutative superassociative algebra
$\mathcal A\ltimes_{-\mathfrak L_{\mu}^{\star}}\mathcal A^{*}$.   
\end{cor}
\begin{proof}
It is similar to the proof of Proposition \ref{comp-H-ant-Zinb-from-Connes-Coc}, by using Eq. \eqref{cond-mult-bil-form}.    
\end{proof}

\subsection{Hom-Lie superalgebras }
\begin{df}(\cite{Faouzi-Abdenacer})
A \textbf{Hom-Lie superalgebra} is a triple $(\mathcal{A}, [\cdot,\cdot],\alpha)$ consisting of a $\mathbb{Z}_2$-graded vector space $\mathcal{A}$, an even bilinear map $[\cdot,\cdot] : \mathcal{A}\otimes \mathcal{A} \longrightarrow \mathcal{A},~~( \;[\mathcal{A}_i,\mathcal{A}_j]\subseteq \mathcal{A}_{i+j},
~~\forall~~i,j\in \mathbb{Z}_2\;)$ and an even linear map  $\alpha:\mathcal{A}\rightarrow\mathcal{A}$  satisfying:
\begin{eqnarray}
 \label{H-skwesym}
 &&[x,y] = -(-1)^{|x||y|}[y,x]\quad \text{( super-skew-symmetry)},\\
\label{H-sJ}
&&  [\alpha(x),[y,z]]=[[x,y],\alpha(z)]+(-1)^{|x||y|}[\alpha(y),[x,z]] \quad \text{(\;Hom super-Jacobi identity\;)},
\end{eqnarray}
$\forall\ x,y,z \in \mathcal{H}(\mathcal{A})$.
A Hom-Lie superalgebra $(\mathcal A,[\cdot,\cdot],\alpha)$ is called \textbf{multiplicative} if $\alpha$ is a Hom-algebras morphism, that is, $\alpha([x,y])=[\alpha(x),\alpha(y)],\;\forall x,y\in\mathcal H(\mathcal A)$ and called \textbf{regular} if $\alpha$ is bijective.
\end{df}
\begin{exa}(\cite{O-N})
Let $\mathcal{A}=\mathcal{A}_{\overline{0}}\oplus \mathcal{A}_{\overline{1}}$ be a $2$-dimensional $\mathbb{Z}_2$-
graded vector space, where $\mathcal{A}_{\overline{0}}=<e_1>$
 and $\mathcal{A}_{\overline{1}}=<e_2>$. We consider the nonzero product given by
$$[e_2,e_2]=2\lambda e_1,$$
and the even linear map $\alpha:\mathcal{A}\rightarrow\mathcal{A}$ defined on the basis of $\mathcal{A}$ by $$\alpha(e_1)=\lambda e_1\;\text{and}\;\alpha(e_2)=\lambda e_2,$$
where $\lambda\in\mathbb{C}$. The triplet $(\mathcal{A},[\cdot,\cdot],\alpha)$ is a Hom-Lie superalgebra.
\end{exa}
\begin{thm}\cite{Faouzi-Abdenacer}\label{Lie-twist-thm}
 Let $(\mathcal A,[\cdot,\cdot])$ be a Lie superalgebra and $\alpha:\mathcal A\to\mathcal A$ be an even algebra endomorphism. Define the linear map $[\cdot,\cdot]_\alpha:\mathcal A\otimes\mathcal A$ by
 \begin{equation}\label{exp-twist-Lie-sup}
  [x,y]_\alpha=[\alpha(x),\alpha(y)],\;\forall x,y\in\mathcal A.   
 \end{equation}
 Then, $(\mathcal A,[\cdot,\cdot]_\alpha,\alpha)$ is a multiplicative Hom-Lie superalgebra.
\end{thm}
\begin{prop}(\cite{Faouzi-Abdenacer})\label{Hom-Lie-super-twist}
Let $(\mathcal{A},\mu,\alpha)$ be a super-commutative Hom-associative superalgebra, then $(\mathcal{A},[\cdot,\cdot]_\mu,\alpha)$ is a Hom-Lie superalgebra, where
$$[x,y]_\mu=\mu(x,y)-(-1)^{|x||y|}\mu(y,x),$$
for all $x,y\in\mathcal{H}(\mathcal{A})$.
\end{prop}

\begin{df}(\cite{Fawzi-Makhlouf-Saadaoui})
A \textbf{representation} of a Hom-Lie superalgebra $(\mathcal{A},[\cdot,\cdot],\alpha)$ on a $\mathbb{Z}_2$-graded vector space $V$ with respect to $\beta\in gl(V)$ is an even linear map $\rho:\mathcal{A}\to gl(V)$ such that for all $x,y\in\mathcal{H}(\mathcal A)$, the following equalities are satisfied:
\begin{eqnarray}
\rho(\alpha(x))\beta&=&\beta\rho(x),\label{repres-Hom-Lie1}\\
\rho([x,y])\beta&=&\rho(\alpha(x))\rho(y)-(-1)^{|x||y|}\rho(\alpha(y))\rho(x).\label{repres-Hom-Lie2}
\end{eqnarray}
This type of representation of a Hom-Lie superalgebra is typically denoted by $(V, \rho, \beta)$.
\end{df}
\begin{exa}
For any $x\in\mathcal{H}(\mathcal A)$, the linear map $\mathfrak{ad}:\mathcal A\to gl(\mathcal A);\;x\mapsto \mathfrak{ad}(x)$ (or $\mathfrak{ad}_x$) defined by
$$\mathfrak{ad}_x(y)=[x,y],\;\forall y\in\mathcal{H}(\mathcal A),$$
defines a representation on $\mathcal A$ called \textbf{adjoint representation}.
\end{exa}
\begin{prop}(\cite{O-N})\label{sumdirecthomLie}
Let $(\mathcal{A}, [\cdot,\cdot ],\alpha)$ be a Hom-Lie superalgebra and $(V,\beta)$ be a Hom-module. Let
$\rho:\mathcal{A} \to gl(V )$ be an even linear map. The triple $(V,\rho,\beta)$ is a representation of
$(\mathcal{A}, [\cdot,\cdot ],\alpha)$ if and only if the direct sum  $\mathcal{A}\oplus V $ of $\mathbb{Z}_2$-graded vector spaces $\mathcal{A}$ and $V$
turns into a Hom-Lie superalgebra by defining the linear map \eqref{Hom-ass-direct-sum2} and the following multiplication
\begin{equation}
    [x+u,y+v]_{\mathcal{A}\oplus V}=[x,y]+\rho(x)v-(-1)^{|y||u|}\rho(y)u\label{Hom-Lie-direct-sum},
\end{equation}
called semi-direct product of the Hom-Lie superalgebra $(\mathcal A,[\cdot,\cdot],\alpha)$ and $V$ which simply denoted by $\mathcal A\ltimes_{\rho} V$.
\end{prop}

Let $\mathcal A$ and $V$ be two $\mathbb Z_2$-graded vector spaces. For an even linear map $\rho:\mathcal A\to gl(V)$, we set $\rho^\ast:\mathcal A\to V^\ast$ by
\begin{equation}\label{n-dual-rep-H-Lie-sup}
    \langle\rho^*(x)(\xi),u\rangle=-\langle\xi,\rho(x)(u)\rangle,\;\forall x\in\mathcal{H}(\mathcal A),\;u\in\mathcal{H}(V),\;\xi\in V^*.
\end{equation}
Here, $\langle\cdot,\cdot\rangle$ is the usual pairing between $V$ and $V^\ast$.
If $(V,\rho,\beta)$ is a representation of a multiplicative Hom-Lie superalgebra $(\mathcal A,[\cdot,\cdot],\alpha)$, then $\rho^*$ is not in general a representation of $\mathcal A$ anymore. Let us define the map $\rho^\star:\mathcal A\to gl(V^\ast)$ by
\begin{align}\label{dual-representation}
    \langle\rho^\star(x)(\xi),v\rangle:&=\langle\rho^\ast(\alpha(x))((\beta^{-2})^{\ast}(\xi)),v\rangle\\\nonumber&=-\langle\xi,\rho(\alpha^{-1}(x))(\beta^{-2}(v))\rangle,
\end{align}
for all $\xi\in V^\ast,\;x\in\mathcal{H}(\mathcal A)$ and $v\in V$. Then, the triplet $(V^\ast,\rho^\star,(\beta^{-1})^*)$  is a representation of $(\mathcal A ,[\cdot,\cdot],\alpha)$ on the dual vector space $V^\ast$ with
respect to the map $(\beta^{-1})^\ast$. This is also known as the “dual representation” of $(V,\rho,\beta)$.\\
In particular, let us also recall that the “coadjoint representation” of a regular Hom-Lie superalgebra $(\mathcal A,[\cdot,\cdot],\alpha)$ on $\mathcal A^\ast$ with respect to $(\alpha^{-1})^\ast$ 
is given by the triplet $(\mathcal A^\ast,\mathfrak{ad}^\star,(\alpha^{-1})^\ast)$, where 
\begin{align}\label{dual-adjoint-rep}
  \langle \mathfrak{ad}^\star(x)(\xi),y\rangle&=-<\xi,[\alpha^{-1}(x),\alpha^{-2}(y)]> ,
\end{align}
for all $x,y\in \mathcal{H}(\mathcal A),\;\xi\in\mathcal A^\ast$.\\

\begin{df}
Let $(\mathcal A,[\cdot,\cdot],\alpha)$ be a Hom-Lie superalgebra. An even bilinear form $\mathfrak{B}:\mathcal A\times\mathcal A\to \mathbb K$ is called:
\begin{enumerate}
\item[(1)] \textbf{supersymmetric} if $\mathfrak{B}(x,y)=(-1)^{|x||y|}\mathfrak{B}(y,x),\,\,\ \forall x,y\in\mathcal{H}(\mathcal A)$
\item[(2)] \textbf{super-skewsymmetric} if $\mathfrak{B}(x,y)=-(-1)^{|x||y|}\mathfrak{B}(y,x),\;\forall x,y\in\mathcal{H}(\mathcal A)$.
\item[(3)] \textbf{Invariant} if 
\begin{equation}\label{inv-bil-form-H-Lie-sup}
 \mathfrak B([x,y],\alpha(z))=\mathfrak B(\alpha(x),[y,z]),\;\forall x,y,z\in\mathcal H(\mathcal A). 
\end{equation}
\item[4)] \textbf{Even},if $\mathfrak B(\mathcal A_{\overline 0},\mathcal A_{\overline 1})=\mathfrak B(\mathcal A_{\overline 1},\mathcal A_{\overline 0})=\{0\}$.
\end{enumerate}
\end{df}
\begin{df}
Let $(\mathcal A,[\cdot,\cdot],\alpha)$ be Hom-Lie superalgebra and $\mathfrak B:\mathcal A\times\mathcal A\to \mathbb K$ be an even  invariant nondegenerate bilinear form satisfying 
Eq. \eqref{cond-mult-bil-form}. The quadruple $(\mathcal A,[\cdot,\cdot],\alpha,\mathfrak B)$ is called quadratic Hom-Lie superalgebra.

If $\alpha$ is an involution, that is $\alpha^2=id$ (resp. invertible), the quadratic Hom-Lie superalgebra is said to be involutive (resp. regular) quadratic
Hom-Lie superalgebra and we write for shortness \textbf{IQH}-Lie superalgebra (resp. \textbf{RQH}-Lie algebra).
\end{df}
Suppose that $(\mathcal A,[\cdot,\cdot],\alpha)$ is a Hom-Lie superalgebra. Then the natural nondegenerate symmetric
bilinear form $\mathfrak B_d$ on $\mathcal A\oplus\mathcal A^\ast$ defined by Eq. \eqref{natur-biln-form-H-assoc-sup} 
is invariant on the Hom-Lie superalgebra $\mathcal A\ltimes_{\mathfrak{ad}^\star}\mathcal A^\ast$.

\section{Hom-anti-Zinbiel superalgebras and Hom-anti-pre-Lie superalgebras}\label{Sec3}

In this section, we investigate Hom-anti-Zinbiel and Hom-anti-pre-Lie superalgebras through the framework of nondegenerate supersymmetric bilinear forms. We extend the corresponding constructions established for the associative and Lie settings to the Hom-superalgebra context. In particular, we introduce the notions of super-commutative Connes cocycles and super-commutative $2$-cocycles, establish their correspondence with Hom-anti-Zinbiel and Hom-anti-pre-Lie superalgebras, respectively, and provide characterizations in terms of anti-super-$\mathcal O$-operators and anti-Rota-Baxter operators. These results provide the algebraic foundation for the construction of compatible Hom-anti-pre-Poisson superalgebras developed in the following section.\\

\subsection{Hom-anti-Zinbiel superalgebras}

\begin{df}
 A \textbf{Hom-anti-Zinbiel superalgebra} is a triplet $(\mathcal A,\star,\alpha)$ consisting of a $\mathbb Z_2$-graded vector space $\mathcal A$, an even bilinear operation $\star:\mathcal A\times\mathcal A\to\mathcal A$ and an even linear map $\alpha:\mathcal A\to\mathcal A$ satisfying for all $x,y,z\in\mathcal H(\mathcal A)$, the following identities:
 \begin{eqnarray}
\alpha(x)\star(y\star z)&=&-(x\star y+(-1)^{|x||y|}y\star x)\star\alpha(z),\label{cond-h-anti-Zinb}\\
\alpha(x)\star(y\star z)&=&(-1)^{|x|(|y|+|z|)}\alpha(y)\star(z\star x).\label{cond-h-anti-Zinb1}     
 \end{eqnarray} 
\end{df}
\begin{prop}\label{equiv-H-ant-Z-H-ass-rep}
Let $\mathcal A$ be a vector space together with a bilinear operation  $\star:\mathcal A\times\mathcal A\to\mathcal A$ be an even bilinear map. Then the following conditions are equivalent:
\begin{enumerate}
\item $(\mathcal A,\star,\alpha)$ is a Hom-anti-Zinbiel algebra.
\item $(\mathcal A,\mu_\star,\alpha)$ with the bilinear operation $\mu_\star$ defined by Eq. \eqref{comp-comm-hom-ass-of-Hom-Zinb} is a super-commutative Hom-associative
superalgebra and $(\mathcal A,-\mathfrak L_\star,\alpha)$ is a representation of $(\mathcal A,\mu_\star,\alpha)$.
\item There is a commutative associative algebra structure on $\mathcal A\oplus\mathcal A$ defined by
\begin{eqnarray}
(\alpha\oplus\alpha)(x,a)&=&(\alpha(x),\alpha(a)),\label{struc-supcomm-H-ass-H-ant-zinb1}\\
\mu_{\mathcal A\oplus\mathcal A}((x,a),(y,b))&=&(x\star y+(-1)^{|x||y|}y\star x,-x\star b-(-1)^{|a||y|}y\star a),\label{struc-supcomm-H-ass-H-ant-zinb2}
\end{eqnarray}
for all $x,y,a,b\in\mathcal H(\mathcal A)$.
\end{enumerate}
 $(\mathcal A,\mu_\star,\alpha)$ is called the \textbf{sub-adjacent super-commutative
Hom-associative superalgebra} of a Hom-anti-Zinbiel superalgebra $(\mathcal A,\star,\alpha)$ and $(\mathcal A,\star,\alpha)$ is called a \textbf{compatible
Hom-anti-Zinbiel superalgebra} of $(\mathcal A,\mu_\star,\alpha)$.
\end{prop}
\begin{proof}
$(1)\Longrightarrow(2)$\\ 
Let $(\mathcal A,\star,\alpha)$ be a Hom-anti-Zinbiel superalgebra. For any homogeneous $x,y\in\mathcal A$, we have 
\begin{align*}
\mu_\star(x,y)&=(x\star y+(-1)^{|x||y|}y\star x)\\&=(-1)^{|x||y|}(y\star x+(-1)^{|x||y|}x\star y)\\&=(-1)^{|x||y|}\mu_\star(y,x). 
\end{align*}
Then $\mu$ is super-commutative.\\
For any $x,y,z\in\mathcal H(\mathcal A)$, by Eq. \eqref{cond-h-anti-Zinb1}, we have
\begin{align*}
 \mu_\star(\alpha(x),\mu_\star(y,z))-\mu_\star(\mu_\star(x,y),\alpha(z))&=\alpha(x)\star(y\star z)+(-1)^{|y||z|}\alpha(x)\star(z\star y)\\&+(-1)^{|x|(|y|+|z|)}\Big((y\star z)\star+(-1)^{|y||z|}(z\star y)\Big)\star\alpha(x)\\&-(x\star y+(-1)^{|x||y|}y\star x)\star\alpha(z)\\&-(-1)^{|z|(|x|+|y|)}\alpha(z)\star(x\star y+(-1)^{|x||y|}y\star x)\\&=2(-1)^{|x|(|y|+|z|)}(y\star z)\star\alpha(x)+2(-1)^{|x|(|y|+|z|)+|y||z|}(z\star y)\star\alpha(x)\\&-2(x\star y)\star\alpha(z)-2(-1)^{|x||y|}(y\star x)\star\alpha(z)\\&=2(\alpha(x)\star(y\star z)-(-1)^{|x|(|y|+|z|)}\alpha(y)\star(z\star x))\\&=0.
\end{align*}
That is, $(\mathcal A,\mu,\alpha)$ is a Hom-associative superalgebra.\\
For any homogeneous elements $x,y,z\in\mathcal A$, since $\star$ Eq. \eqref{cond-h-anti-Zinb}, therfore
\begin{align*}
-\mathfrak L_\star(\mu_\star(x,y))\alpha(z)&=-(x\star y)\star\alpha(z)-(-1)^{|x||y|}(y\star x)\star\alpha(z)\\&=\alpha(x)\star(y\star z)\\&=\mathfrak L_\star(\alpha(x))(\mathfrak L_\star(y)z).     
\end{align*}
Then $-\mathfrak L_\star$ satisfies Eq. \eqref{rAs2} which implies that $(\mathcal A,-\mathfrak L_\star,\alpha)$ is a representation of $(\mathcal A,\mu_\star,\alpha)$.\\
$(2)\Longrightarrow(1)$\\
Suppose that $(\mathcal A,\mu_\star,\alpha)$  is a super-commutative Hom-associative
superalgebra and $(\mathcal A,-\mathfrak L_\star,\alpha)$ is a representation of $(\mathcal A,\mu_\star,\alpha)$, then for any $x,y,z\in\mathcal H(\mathcal A)$, we have
\begin{align*}
-(x\star y+(-1)^{|x||y|}y\star x)\star\alpha(z)&=-\mu_\star(x,y)\alpha(z)\\&=-\mathfrak L_\star(\mu(x,y))\alpha(z)\\&=(-\mathfrak L_\star)(\alpha(x))((-\mathfrak L_\star)(y)z)=\mathfrak L_\star(\alpha(x))(\mathfrak L_\star(y)z)\\&=\alpha(x)\star(y\star z).    
\end{align*}
Then, the bilinear map $\star$ satisfies Eq. \eqref{cond-h-anti-Zinb}. In addition, for any $x,y,z\in\mathcal H(\mathcal A)$, since $(\mathcal A,\mu_\star,\alpha)$ is a super-commutative Hom-associative superalgebra, we have
\begin{align*}
0&=\mu_\star(\alpha(x),\mu_\star(y,z))-\mu_\star(\mu_\star(x,y),\alpha(z))\\&=2\Big(\alpha(x)\star(y\star z)-(-1)^{|x|(|y|+|z|)}\alpha(y)\star(z\star x)\Big) .   
\end{align*}
Therefore $\star$ satisfies Eq. \eqref{cond-h-anti-Zinb1}, and thus it can be known that $(\mathcal A,\star,\alpha)$ is a Hom-anti-Zinbiel superalgebra.\\
$(1)\Longrightarrow(3)$ Straightforward computation.
\end{proof}
\begin{df}
Let $(\mathcal A,\mu,\alpha)$ be a super-commutative Hom-associative  superalgebra and $(V,\eta,\beta)$ be a representation. An even linear map $T:V\to\mathcal A$ is called an \textbf{anti-super-$\mathcal O$-operator} on $(\mathcal A,\mu,\alpha)$ associate with $(V,\eta,\beta)$ if the following equations holds:
\begin{eqnarray}
T\circ\beta&=&\alpha\circ T,\label{cond-ant-O-oper-ncomm-Hom-ass1} \\
\mu( T(u),T(v))&=&-T\big(\eta(T(u))v+(-1)^{|u||v|}\eta(T(v))u\big),\;\forall u,v\in\mathcal H(V).\label{cond-ant-O-oper-ncomm-Hom-ass2}
\end{eqnarray}
Furthermore, $T$ is called \textbf{strong} if
\begin{equation}\label{strong-ant-O-oper-ncomm-Hom-ass}
\eta(\mu(T(u),T(v)))\beta(w)= (-1)^{|u|(|v|+|w|)}\eta(\mu(T(v),T(w)))\beta(u),\;\forall u,v,w\in\mathcal{H}(V).   
\end{equation}
In particular, a super-anti-$\mathcal O$-operator $\mathcal R$ of $(\mathcal A,\mu,\alpha)$ associated with the bimodule $(\mathcal A,\mathfrak L,\mathcal R,\alpha)$
is called an \textbf{anti-Rota-Baxter operator (of weight zero)}; that is, $\mathcal R :\mathcal A\to\mathcal A $ is an even linear map commuting with $\alpha$ and satisfying
\begin{equation}\label{ant-RB-oper-noncomm-ass}
 \mu(\mathcal R(x),\mathcal R(y))=-\mathcal R\big(\mu(\mathcal R(x),y)+\mu(x,\mathcal R(y))\big),\;\forall x,y\in\mathcal A.   
\end{equation}
An anti-Rota–Baxter operator $\mathcal R$ is called \textbf{strong} if $\mathcal R$ satisfies
\begin{equation}\label{strong-ant-RB-oper-noncomm-ass}
\mu(\mu(\mathcal R(x),\mathcal R(y)),\alpha(z))=\mu(\alpha(x),\mu(\mathcal R(y),\mathcal R(z)),\;\forall x,y,z\in\mathcal A.    
\end{equation}
\end{df} 

\begin{thm}\label{Hom-ant-Zinb-by-ant-O-oper}
Let $T:V\rightarrow\mathcal A$ be an anti-super-$\mathcal O$-operator on a super-commutative Hom-associative superalgebra $(\mathcal A,\mu,\alpha)$ with respect to a representation $(V,\eta,\beta)$. Let us define a binary operation $\star_T:V\times V\to V$ by:
\begin{equation}\label{H-anti-Zinb-from-O-oper-H-ass}  
u \star_T v=-\eta(T(u))v,\;\forall u,v\in\mathcal H(V).
\end{equation}
Then $(V,\star_T,\beta)$ is a Hom-anti-Zinbiel superalgerbra if and only if $T$ is strong. In this case $T$ is a homomorphism of super-commutative Hom-associative superalgebras from $(V,\mu_{\star_T},\beta)$ to $(\mathcal A,\mu,\alpha)$. Furthermore, there is an induced Hom-anti-Zinbiel superalgebra structure on
$T(V)=\{T(u),\;u\in V\}\subseteq A$ given by
\begin{equation}\label{ant-Hom-Zinb-T(V)}
T(u)\star T(v)=T(u\star_Tv),\;\forall u,v\in\mathcal H(V).
\end{equation}
\end{thm}
\begin{proof}
For any homogeneous elements $u,v,w\in V$, we have:
\begin{align*}
\beta(u)\star_T(v\star_T w)&=\eta(T(\beta(u)))(\eta(T(v))(w))\\&= \eta(\alpha T(u))(\mu(T(v))(w))\\&= \eta(\mu(T(u),T(v)))\beta(w)\\&=-\eta\Big(T\big(\eta(T(u))v+(-1)^{|u||v|}\eta(T(v))u\big)\Big)\beta(w)\\&=-\Big[\eta\Big(T(\eta(T(u))v)\Big)\beta(w)+(-1)^{|u||v|}\eta\Big(T(\eta(T(v))u)\Big)\beta(w)\Big]\\&=-(u\star_T v+(-1)^{|u||v|}v\star_T u)\star_T\beta(w).  
\end{align*}
Then, $\star_T$ satisfies Eq. \eqref{cond-h-anti-Zinb} on $V$.\\
By the same way, $T$ is strong if and only if
\begin{align*}
\beta(u)\star_T(v\star_T w)&=\eta(\mu(T(u),T(v)))\beta(w)\\&\overset{(\ref{strong-ant-O-oper-ncomm-Hom-ass})}{=}  (-1)^{|u|(|v|+|w|)}\eta(\mu(T(v),T(w)))\beta(u)\\&= (-1)^{|u|(|v|+|w|)}  \beta(v)\star_T(w\star_T u).
\end{align*}
Therefore, Eq. \eqref{cond-h-anti-Zinb} also satisfies which gives the Proof.
\end{proof}

\begin{cor}\label{H-ant-Zinb-to-Hom-ass-RB}
Let $\mathcal R$ be a strong anti-Rota-Baxter operator of weight zero on an Hom-associative superalgebra  $(\mathcal A,\mu,\alpha)$. Then, $(\mathcal A,\star_\mathcal R,\alpha)$ where $\star_\mathcal R:\mathcal A\to\mathcal A$ is defined by 
\begin{equation}\label{H-anti-dend-From-ARB-Hass} 
 x\star_\mathcal R y=-\mu(\mathcal R(x),y),
\end{equation}
for any $x,y\in\mathcal H(\mathcal A)$ is a Hom-anti-Zinbiel superalgebra.
\end{cor}
\begin{lem}\label{inv-O-op-H-ass-aut-strong}
An invertible anti-super-$\mathcal O$-operator of a super-commutative Hom-associative superalgebra is automatically
strong.    
\end{lem}
\begin{proof}
 Let $T:V \to \mathcal A$ be an invertible anti-super-$\mathcal O$-operator on a super-commutative Hom-associative superalgebra
 $(\mathcal A,\mu,\alpha)$ with respect to a representation 
$(V,\eta,\beta)$.
\begin{align*}
    \mu_{\star_T}(u,v)&=u\star_T v+(-1)^{|u||v|}v\star_T u \\
                    &=-\eta(T(u))v-(-1)^{|u||v|}\eta(T(v))u\\
                    &=T^{-1}(\mu(T(u),T(v)))
\end{align*}

\begin{align*}
 \mu_{\star_T}(\mu_{\star_T}(u,v),\beta(w))&= \mu_{\star_T} T^{-1}
 (\mu(T(u),T(v)),\beta(w))\\
 &=T^{-1}(\mu(\mu(T(u),T(v),T(\beta(w))))) \\
 &=T^{-1}(\mu(\mu(\alpha(T(u)),\mu(T(v),T(w))),\alpha(T(w))))\\
 &=\mu_{\star_T}(\beta(u),T^{-1}(\mu(T(u),T(v))) \\
 &=\mu_{\star_T}(\beta(u),\mu_{\star_T}(v,w))
\end{align*}
 
  It follows that $(V,\mu_{\star_T},\beta)$ is a Hom-associative superalgebra. Applying Proposition \ref{equiv-H-ant-Z-H-ass-rep}, we conclude that $(V,\star_T,\beta)$ is a Hom-anti-Zinbiel superalgebra. Therefore, Theorem \ref{Hom-ant-Zinb-by-ant-O-oper} ensures that $T$ is strong. 
\end{proof}
\begin{thm}\label{comp-ant-dend-iif-inver-O-op}
Let $(\mathcal A,\mu,\alpha)$ be a super-commutative Hom-associative superalgebra. Then, there is a compatible Hom-anti-Zinbiel superalgebra structure on $\mathcal A$ if and only if there exists an invertible anti-super-$\mathcal O$-operator on $(\mathcal A,\mu,\alpha)$.  
\end{thm}
\begin{proof}
  Suppose that $(\mathcal{A},\star,\alpha)$ is a compatible Hom-anti-Zinbiel superalgebra on $(\mathcal{A},\mu,\alpha)$ .Then 
\begin{align*}
  \mu(x,y)&=x\star y+(-1)^ {|x||y|} y\star x \\
          &=-(-\mathfrak L_\star(x)y-(-1)^ {|x||y|}\mathfrak L_\star(y)x) 
\end{align*}
It follows that, the identity map $T=id$ is an invertible anti-super-$\mathcal O$-operator of $(\mathcal{A},\mu,\alpha)$ with respect to the representation $(\mathcal{A},-\mathfrak L_\star,\alpha)$. \\
Conversely, let $T\to\mathcal{A}$ be an invertible anti-super-$\mathcal{O}$-operator on the Hom-associative superalgebra $(\mathcal{A},\mu,\alpha)$ associated with the representation $(V,\eta,\beta)$. By Lemma \ref{inv-O-op-H-ass-aut-strong}, the operator $T$ is strong. Hence, Theorem \ref{Hom-ant-Zinb-by-ant-O-oper} ensures that $V$ carries a Hom-anti-Zinbiel superalgebra structure whose multiplication is defined by Eq.~\eqref{H-anti-Zinb-from-O-oper-H-ass}.

Since $T$ is invertible, this Hom-anti-Zinbiel superalgebra structure can be transferred to the underlying vector space of $\mathcal{A}$. More precisely, the induced multiplication on $\mathcal{A}$ is given by Eq.~\eqref{ant-Hom-Zinb-T(V)}. For arbitrary homogeneous elements $x,y\in\mathcal{H}(\mathcal{A})$, there exist homogeneous elements $u,v\in\mathcal{H}(V)$ such that
$x=T(u)$ and $ y=T(v)$.

Therefore,
\begin{align*}
x\star y
&=T(u)\star T(v)\\
&=T(u\star_Tv)\\
&=T\bigl(-\eta(x)T^{-1}(y)\bigr).
\end{align*}

On the other hand, by Eq.~\eqref{cond-ant-O-oper-ncomm-Hom-ass2}, we obtain
\begin{align*}
\mu(x,y)
&=\mu\bigl(T(u),T(v)\bigr)\\
&=-T\bigl(\eta(T(u))v+(-1)^{|u||v|}\eta(T(v))u\bigr)\\
&=-T\bigl(\eta(x)T^{-1}(y)+(-1)^{|x||y|}\eta(y)T^{-1}(x)\bigr)\\
&=x\star y+(-1)^{|x||y|}y\star x.
\end{align*}
Consequently, the multiplication $\star$ is compatible with the Hom-associative product $\mu$, and thus $(\mathcal{A},\star,\alpha)$ defines a compatible Hom-anti-Zinbiel superalgebra structure on the Hom-associative superalgebra $(\mathcal{A},\mu,\alpha)$.
\end{proof}

\begin{df}\label{sup-comm-Connes-coc}
 Let $(\mathcal A,\mu,\alpha)$ be a super-commutative Hom-associative superalgebra and $\mathfrak B$ be an even bilinear form on $(\mathcal A,\mu,\alpha)$. If $\mathfrak B$ is supersymmetric and satisfies Eq. \eqref{def-con-cocy-supcomm-H-ass-sup} then $\mathfrak B$ is called a \textbf{super-commutative Connes
cocycle}.   
\end{df}
\begin{thm}\label{from-h-ass-to-H-anti-Zinb-via-comm-Connes-coc}
 Let $(\mathcal A,\mu,\alpha)$ be a multiplicative super-commutative Hom-associative superalgebra and $\mathfrak B$ be a nondegenerate
super-commutative Connes cocycle on $(\mathcal A,\mu,\alpha)$. Then, there exists a compatible
Hom-anti-Zinbiel superalgebra structure $(\mathcal A,\star,\alpha)$ on $(\mathcal A,\mu,\alpha)$ defined by 
\begin{equation}\label{eq-sup-comm-con-cocy-to-H-ant-Zinb}
\mathfrak{B}(\alpha(x)\star\alpha(y),\alpha^2(z))
=-(-1)^{|x||y|}\mathfrak{B}(\alpha(y),\mu(x,z)), \forall x,y,z\in\mathcal H(\mathcal A). 
\end{equation}
Conversely, let $(\mathcal A,\star,\alpha)$ be a Hom-anti-Zinbiel superalgebra and $(\mathcal A,\mu,\alpha)$ be the sub-adjacent super-commutative
Hom-associative superalgebra. Then the natural nondegenerate symmetric bilinear form $\mathfrak B_d$ defined by Eq. \eqref{natur-biln-form-H-assoc-sup} is a super-commutative Connes cocycle on the super-commutative Hom-associative superalgebra $\mathcal A\ltimes_{-\mathfrak L_{\star}^{\star}}\mathcal A^{*}$.
 
\end{thm}
\begin{proof}
Let $T: \mathcal A\rightarrow \mathcal A^{*}$ be a linear map defined by Eq. \eqref{eq:Hom-ass-T}. Since $\mathfrak B$ is nondegenerate, the map $T$ is invertible. Then for any $a^{*},b^{*}\in \mathcal A^{*}$, there exist $x,y\in\mathcal A$ such that $a^{*}=T(x)$ and $b^{*}=T(y)$. We compute  $$\mathfrak B(\mu(x,y),\alpha(z)) =\langle T(\mu(x,y)),\alpha(z)\rangle =\langle T(\mu(T^{-1}(a^{*}),T^{-1}(b^{*}))),\alpha(z)\rangle,\;\forall z\in\mathcal H(\mathcal A).$$
By Eq. \eqref{dual-adjoint-rep}, we have
\begin{eqnarray*}
\mathfrak B(\mu(y,z),\alpha(x)) &=&(-1)^{|x|(|y|+|z|)}\langle T(\alpha(x)),\mu(y,z)\rangle =-(-1)^{|x||z|}\langle \mathfrak{L}^{\star}(\alpha(y))T(\alpha(x)),\alpha^2(z)\rangle\\
&=&-(-1)^{|x||z|}\langle \mathfrak{L}^{\star}(\alpha(y))(\alpha^{-1})^*T(x),\alpha^2(z)\rangle\\
&=&-(-1)^{|x||z|}\langle (\alpha^{-1})^*\mathfrak{L}^{\star}(y)T(x),\alpha^2(z)\rangle\\
&=&-(-1)^{|x||z|}\langle (\mathfrak{L}^{\star}(y)T(x),\alpha(z)\rangle\\
&=&-(-1)^{|a||z|}\langle \mathfrak{L}^{\star}(T^{-1}(b^{*}))a^{*},\alpha(z)\rangle.
\end{eqnarray*}
Similarly we have
$$\mathfrak B(\mu(z,x),\alpha(y))=(-1)^{|x||z|}\mathfrak B(\mu(x,z),\alpha(y))=-(-1)^{|z|(|a|+|b|)}\langle \mathfrak{L}^{\star}(T^{-1}(a^{*}))b^{*},\alpha(z)\rangle.$$ 
   
Since $\mathfrak B$ is a super-commutative Connes cocycle on the Hom-associative superalgebra $(\mathcal A,\mu,\alpha)$, we have \[ \mathfrak B(\mu(x,y),\alpha(z)) +(-1)^{|x|(|y|+|z|)}\mathfrak B(\mu(y,z),\alpha(x)) +(-1)^{|z|(|x|+|y|)}\mathfrak B(\mu(z,x),\alpha(y))=0 . \] Therefore \[ \mu(T^{-1}(a^{*}),T^{-1}(b^{*})) = -T^{-1}\big(-\mathfrak{L}^{\star}(T^{-1}(a^{*}))(b^{*}) -(-1)^{|a||b|} \mathfrak{L}^{\star}(T^{-1}(b^{*}))(a^{*})  \big). \] Thus $T^{-1}$ is an invertible anti-super-$\mathcal O$-operator of the Hom-associative superalgebra $(\mathcal A,\mu,\alpha)$ associated with the representation $(\mathcal A^{*},-\mathfrak{L}^{\star},(\alpha^{-1})^{*})$ and Lemma \ref{inv-O-op-H-ass-aut-strong} gives that $T$ is strong. Then, by using Corollary \ref{H-ant-Zinb-to-Hom-ass-RB}, there exists a compatible Hom-anti-Zinbiel superalgebra structure $\star$ on $(\mathcal A,\mu,\alpha)$ defined by \begin{equation}\label{eq:Hom-assoc-product} x\star y=-T^{-1}\big(-\mathfrak{L}^{\star}(x)T(y)\big),\qquad \forall x,y\in\mathcal H(\mathcal A). \end{equation} Finally, for any $x,y,z\in\mathcal H(\mathcal A)$, we obtain 
\begin{eqnarray*}
\mathfrak B(\alpha(x)\star\alpha( y),\alpha^2(z)) &=&\langle T(\alpha(x)\star\alpha( y)),\alpha^2(z)\rangle\\&=&\langle \mathfrak{L}^{\star}(\alpha(x))T(\alpha(y)),\alpha^2(z)\rangle\\&=&-(-1)^{|x||y|}\langle T(\alpha(y)),\mu(x,z)\rangle \\&=&-(-1)^{|x||y|}\mathfrak B(\alpha(y),\mu(x,z)). 
    \end{eqnarray*}
    Hence the Hom-anti-Zinbiel product $\star$ satisfies \[ \mathfrak B(\alpha(x)\star\alpha( y),\alpha^2(z))=(-1)^{|x||y|}\mathfrak B(\alpha(y),\mu(x,z)),\; \forall x,y,z\in\mathcal H(\mathcal A), \] which gives the result.
    
\end{proof}

\begin{cor}\label{H-ant-pre-Lie-from-bil-Form-mult-H-Lie}
  Let $(\mathcal A,\mu,\alpha)$ be a multiplicative super-commutative Hom-associative superalgebra and $\mathfrak B$ be a nondegenerate
super-commutative Connes cocycle on $(\mathcal A,\mu,\alpha)$ satisfying Eq. \eqref{cond-mult-bil-form}. Then, there exists a compatible
Hom-anti-Zinbiel superalgebra structure $(\mathcal A,\star,\alpha)$ on $(\mathcal A,\mu,\alpha)$ defined by 
\begin{equation}\label{eq-mult-sup-comm-con-cocy-to-H-ant-Zinb}
\mathfrak{B}(x\star y,\alpha(z))
=-(-1)^{|x||y|}\mathfrak{B}(\alpha(y),\mu(x,z)),\;\forall x,y,z\in\mathcal H(\mathcal A). 
\end{equation}
Conversely, let $(\mathcal A,\star,\alpha)$ be a Hom-anti-Zinbiel superalgebra and $(\mathcal A,\mu,\alpha)$ be the sub-adjacent super-commutative
Hom-associative superalgebra. Then the natural nondegenerate symmetric bilinear form $\mathfrak B_d$ defined by Eq. \eqref{natur-biln-form-H-assoc-sup} is a super-commutative Connes cocycle on the super-commutative Hom-associative superalgebra $\mathcal A\ltimes_{-\mathfrak L_{\star}^{\star}}\mathcal A^{*}$.
\end{cor}
\subsection{Hom-anti-pre-Lie superalgebras}
\begin{df}(\cite{wi-oth})
A \textbf{Hom-anti-pre-Lie superalgebra} (also called \textbf{Hom-anti-left-symmetric superalgebra}) is a triplet $(\mathcal A,\circ,\alpha)$ consisting of a $\mathbb Z_2$-graded vector space $\mathcal{A}$, an even bilinear map $\circ:\mathcal A\times\mathcal A\to\mathcal A$ and an even linear map $\alpha:\mathcal A\to\mathcal A$ such that, the following conditions hold   \begin{eqnarray}
 &&\alpha(x)\circ(y\circ z)-(-1)^{|x||y|} \alpha(y)\circ(x\circ z)=[y,x]_\circ \circ \alpha(z),\label{cond-Hom-ant-pre-Lie1}  \\
 &&(-1)^{|x||z|}[x,y]_\circ\circ \alpha(z)+(-1)^{|x||y|}[y,z]_\circ\circ \alpha(x)+(-1)^{|y||z|}[z,x]_\circ\circ \alpha(y)=0\label{cond-Hom-ant-pre-Lie2},
\end{eqnarray} 
for any $x,y,z\in\mathcal{H}(\mathcal{A})$, where $[x,y]_\circ=x\circ y-(-1)^{|x||y|}y\circ x$.\\
A Hom-anti-pre-Lie superalgebra $(\mathcal A,\circ,\mu)$ is called \textbf{regular} if $\alpha$ is bijective, and called \textbf{multiplicative} if $\alpha$ is an algebra endomorphism, i.e., $\alpha(x\circ y)=\alpha(x)\circ\alpha(y)$ for any $x,y\in\mathcal A$.
\end{df}
\begin{prop}(\cite{wi-oth})
Let $\mathcal A$ be a $\mathbb Z_2$-grade vector space with an even bilinear map $\circ:\mathcal A\times\mathcal A\to\mathcal A$ and an even linear map $\alpha:\mathcal A\to\mathcal A$. Then the
following assertions are equivalent: 
\begin{enumerate}
\item $(\mathcal A,\circ,\alpha)$ is a Hom-anti-pre-Lie superalgebra.
\item For $(\mathcal A,\circ,\alpha)$, Eq. \eqref{cond-Hom-ant-pre-Lie1} and for any $x,y,z\in\mathcal H(\mathcal A)$, the following equation hold
\begin{equation}\label{cond-equiv1}
\circlearrowleft_{x,y,z}(-1)^{|x||z|}\alpha(x)\circ[y,z]_\circ=(-1)^{|x||z|}\alpha(x)\circ[y,z]_\circ+ (-1)^{|x||y|}\alpha(y)\circ[z,x]_\circ+(-1)^{|y||z|}\alpha(z)\circ[x,y]_\circ=0.
\end{equation}
\item $(\mathcal A,\circ,\alpha)$ is a Hom-Lie admissible superalgebra, such that $(\mathcal A,-\mathfrak L_\circ,\alpha)$ is a representation of the sub-adjacent Hom-Lie superalgebra $(\mathfrak g(\mathcal A),[\cdot,\cdot]_\circ,\alpha)$, where $\mathfrak L_\circ:\mathfrak g(\mathcal A)\to End(\mathcal A)$ is an even linear map defined by $\mathfrak L_\circ(x)(y)=x\circ y,\;\forall x,y\in\mathcal H(\mathcal A)$.
\end{enumerate}
\end{prop}
\begin{thm}\label{Hom-ant-pre-Lie-by-ant-O-oper}(\cite{wi-oth})
 Let $T$ be an anti-super-$\mathcal O$-operator on a Hom-Lie superalgebra $(\mathfrak g,[\cdot,\cdot],\alpha)$ with respect to a representation $(V,\rho,\beta)$. Then, $(V,\circ_V,\beta)$ where, $\circ_V:V\times V\to V$ is defined by
 \begin{equation}\label{ant-O-oper-to-anti-pre-Lie}
 u\circ_V v=-\rho(T(u))v,\;\forall u,v\in V,    
 \end{equation}
 satisfies Eq. \eqref{cond-Hom-ant-pre-Lie1}. Moreover, $(V,\circ_V,\beta)$ is a Hom-Lie-admissible such that $(V,\circ_V,\beta)$ is a Hom-anti-pre-Lie superalgebra if and only if $T$ is strong. So, $T$ is a homomorphism of Hom-Lie superalgebras from the sub-adjacent Hom-Lie superalgebra $(\mathfrak g(V),[\cdot,\cdot]_{\circ_V},\beta)$ to $(\mathfrak g,[\cdot,\cdot],\alpha)$. Therefore, there is an induced Hom-anti-pre-Lie
superalgebra structure on $T(V)=\{T(u);\;u\in V\}\subset\mathfrak g$ given by
\begin{equation}\label{ind-Hom-anti-T(V)}
T(u)\circ_\mathfrak gT(v)=T(u\circ_Vv),\forall u,v\in V,
\end{equation}
and $T$ is a homomorphism of Hom-anti-pre-Lie superalgebras.
\end{thm}
\begin{cor}\label{Hom-ant-pre-Lie-by-ant-R-B}(\cite{wi-oth})
Let $\mathcal R:\mathfrak g\to\mathfrak g$ be a strong anti-Rota-Baxter operator of a Hom-Lie superalgebra $(\mathfrak g,[\cdot,\cdot],\alpha)$. Then, the product 
\begin{equation}\label{ant-R-B-to-anti-pre-Lie}
x\circ y=-[\mathcal R(x),y],\;\forall x,y\in\mathfrak g,    
\end{equation}
defines on $\mathfrak g$ a Hom-anti-pre-Lie superalgebra structure. Conversely, if $\mathcal R:\mathfrak g\to\mathfrak g$ is a linear transformation of a Hom-Lie superalgebra $(\mathfrak g,[\cdot,\cdot],\alpha)$ such that, the product defined by Eq. \eqref{ant-R-B-to-anti-pre-Lie} defines on $\mathcal A$ a Hom-anti-pre-Lie superalgebra, then $\mathcal R$ satisfies Eq. \eqref{cond-strong-Rota-Baxter-Oper} and the following equation:
\begin{equation}\label{Cor-R-B-Hom-anti-pre-Lie}
[[\mathcal R(x),\mathcal R(y)]+\mathcal R([\mathcal R(x),y]+[x,\mathcal R(y)]),\alpha(z)]=0,\forall x,y,z\in\mathfrak g.
\end{equation}
\end{cor}    
\begin{df}(\cite{wi-oth})
Let $(V,\rho,\beta)$ be a representation of a Hom-Lie superalgebra $(\mathcal A,[\cdot,\cdot],\alpha)$. An even linear map $T:V\to\mathcal A$ is called an \textbf{anti-super-$\mathcal{O}$-operator} on $(\mathcal A,[\cdot,\cdot],\alpha)$ associated with $(V,\rho,\beta)$ if it satisfies
\begin{eqnarray}
T\circ\beta&=&\alpha\circ T,\label{cond-ant-O-oper-Hom-Lie1} \\
\lbrack T(u),T(v)\rbrack&=&T\big((-1)^{|u||v|}\rho(T(v))u-\rho(T(u))v\big),\;\forall u,v\in\mathcal H(V).\label{cond-ant-O-oper-Hom-Lie2}
\end{eqnarray}
an anti-super-$\mathcal O$-operator $T$ is called \textbf{strong} if it satisfies:
\begin{equation}\label{cond-strong-O-Oper}
\rho([T(u),T(v)])\beta(w)+(-1)^{|u|(|v|+|w|)}\rho([T(v),T(w)])\beta(u)+(-1)^{|w|(|u|+|v|)}\rho([T(w),T(u)])\beta(v)=0,  
\end{equation}
for all $ u,v,w\in\mathcal H(V)$.\\
In particular, an anti-super-$\mathcal O$-operator $\mathcal R$ of $(\mathcal A,[\cdot,\cdot],\alpha)$ associated with the adjoint representation $(\mathcal A,\mathfrak{ad},\alpha)$ is called an \textbf{anti-Rota-Baxter operator (of weight zero)}, that is, $\mathcal R:\mathcal A\to\mathcal A$ is an even linear map commuting with $\alpha$ satisfying:
\begin{equation}\label{cond-Rota-Baxter-Oper}
 [\mathcal R(y),\mathcal R(x)]= (-1)^{|x||y|}\mathcal R([\mathcal R(x),y]+[x,\mathcal R(y)]),\;\forall x,y\in\mathcal H(\mathcal A).  
\end{equation}
An anti-Rota-Baxter operator $\mathcal R$ is called \textbf{strong}, if it satisfies
\begin{equation}\label{cond-strong-Rota-Baxter-Oper}
[[\mathcal R(x),\mathcal R(y)],\alpha(z)]+(-1)^{|x|(|y|+|z|)} [[\mathcal R(y),\mathcal R(z)],\alpha(x)]+(-1)^{|z|(|x|+|y|)}[[\mathcal R(z),\mathcal R(x)],\alpha(y)]=0,   
\end{equation}
for all $x,y,z\in\mathcal H(\mathcal A)$.
\end{df}

\begin{prop}\label{inver-O-oper}
 An invertible super-anti-$\mathcal{O}$-operator of a Hom-Lie superalgebra is automatically strong.   
\end{prop}
\begin{proof}
Let$(\mathcal A,[\cdot,\cdot],\alpha)$ be a Hom-Lie superalgebra with representation $(V,\rho,\beta)$ and $T:V\to\mathcal A$ be an anti-super-$\mathcal{O}$-operator. Define a bilinear map $\circ_V:V\times V\to V$ by Eq \eqref{ant-O-oper-to-anti-pre-Lie}.Then
\begin{align*}
 [u,v]_{\circ_V}&=u_{\circ_V} v-(-1)^{|u||v|}v_{\circ_V} u  \\
 &=(-1)^{|u||v|}\rho(T(v))u-\rho(T(u))v \\
 &=T^{-1}([T(u),T(v)])
\end{align*}
Furthermore,
\begin{align*}
 [[u,v]_{\circ_V},\beta(w)]_{\circ_V}&=[T^{-1}([T(u),T(v)]),\beta(w)]_{\circ_V} \\
 &=T^{-1}([[T(u),T(v)],\alpha(T(w))])
\end{align*}
Hence, $(V,[\cdot,\cdot]_{\circ_ V},\beta)$ is a Hom-Lie superalgebra.It follows that,$(V,\circ_V,\beta)$ is a Hom-Lie-admissible. Applying Theorem \ref{Hom-ant-pre-Lie-by-ant-O-oper}, we conclude that $T$ is a strong anti-super-$\mathcal{O}$-operator.
\end{proof}
\begin{df}
Let $(\mathcal A,[\cdot,\cdot],\alpha)$ be a Hom-Lie superalgebra, and $\mathfrak B$ an even
super-commutative bilinear form on $\mathcal A$, such that for any $x,y,z\in \mathcal H(\mathcal A)$, the following condition hold:.
\begin{equation}\label{sup-comm-2-cocyc}
\mathfrak B(\alpha(x),[y,z])
+(-1)^{|x|(|y|+|z|)}\mathfrak B(\alpha(y),[z,x])
+(-1)^{|z|(|x|+|y|)}\mathfrak B(\alpha(z),[x,y])=0.
\end{equation}
Then the bilinear form $\mathfrak B$ is called a
\textbf{super-commutative $2$-cocycle} on the Hom-Lie superalgebra $\mathcal A$.    
\end{df}
\begin{thm}\label{ant-h-pr-Lie-from-nondeg-bil}
  Let $\mathfrak{B}$ be a nondegenerate super-commutative $2$-cocycle on a multiplicative Hom-Lie superalgebra $(\mathcal A,[\cdot,\cdot],\alpha)$. Then there exists a compatible Hom-anti-pre-Lie superalgebra structure $\circ$ on $(\mathcal A,[\cdot,\cdot],\alpha)$ given by \begin{equation}\label{inv-Hom-anti-pre-Lie}\mathfrak{B}(\alpha(x)\circ \alpha(y),\alpha^2(z))=(-1)^{|x||y|}\mathfrak{B}(\alpha(y),[x,z]), \;\;\forall x,y,z\in \mathcal H(\mathcal A). 
  \end{equation} 
  Conversely, let $(\mathcal A,\circ,\alpha)$ be a Hom-anti-pre-Lie superalgebra and $(\mathcal A, [\cdot,\cdot],\alpha)$ be the sub-adjacent Hom-Lie
superalgebra. Then the natural nondegenerate symmetric bilinear form $\mathfrak B_d$ defined by Eq. \eqref{natur-biln-form-H-assoc-sup} is a super-commutative $2$-cocycle on the Hom-Lie superalgebra $\mathcal A\ltimes_{-\mathfrak L_\circ^\star}\mathcal A^\ast$.
\end{thm}
\begin{proof}
Define a linear map $T: \mathcal A\rightarrow \mathcal A^{*}$ by \begin{equation}\label{eq:HomT} \langle T(x),y\rangle=\mathfrak{B}(x,y),\qquad \forall x,y\in \mathcal H(\mathcal A). \end{equation} Since $\mathfrak B$ is nondegenerate, the map $T$ is invertible. Let $a^{*},b^{*}\in \mathcal A^{*}$. Then there exist $x,y\in\mathcal H(\mathcal A)$ such that $a^{*}=T(x)$ and $b^{*}=T(y)$. For any $z\in\mathcal H(\mathcal A)$, we compute  $$\mathfrak B([x,y],\alpha(z)) =\langle T([x,y]),\alpha(z)\rangle =\langle T([T^{-1}(a^{*}),T^{-1}(b^{*})]),\alpha(z)\rangle.$$
By Eq.\eqref{dual-adjoint-rep}, we have
\begin{eqnarray*}
\mathfrak B([y,z],\alpha(x)) &=&(-1)^{|x|(|y|+|z|)}\langle T(\alpha(x)),[y,z]\rangle =-(-1)^{|x||z|}\langle \mathfrak{ad}^{\star}(\alpha(y))T(\alpha(x)),\alpha^2(z)\rangle\\
&=&-(-1)^{|x||z|}\langle \mathfrak{ad}^{\star}(\alpha(y))(\alpha^{-1})^*T(x),\alpha^2(z)\rangle\\
&=&-(-1)^{|x||z|}\langle (\alpha^{-1})^*\mathfrak{ad}^{\star}(y)T(x),\alpha^2(z)\rangle\\
&=&-(-1)^{|x||z|}\langle (\mathfrak{ad}^{\star}(y)T(x),\alpha(z)\rangle\\
&=&-(-1)^{|a||z|}\langle \mathfrak{ad}^{\star}(T^{-1}(b^{*}))a^{*},\alpha(z)\rangle.
\end{eqnarray*}
Similarly we have
$$\mathfrak B([z,x],\alpha(y))=-(-1)^{|x||z|}\mathfrak B([x,z],\alpha(y))=(-1)^{|z|(|a|+|b|)}\langle \mathfrak{ad}^{\star}(T^{-1}(a^{*}))b^{*},\alpha(z)\rangle.$$ 
   
Since $\mathfrak B$ is a super-commutative $2$-cocycle on the Hom-Lie superalgebra $(\mathcal A,[\cdot,\cdot],\alpha)$, we have \[ \mathfrak B([x,y],\alpha(z)) +(-1)^{|x|(|y|+|z|)}\mathfrak B([y,z],\alpha(x)) +(-1)^{|z|(|x|+|y|)}\mathfrak B([z,x],\alpha(y))=0 . \] Therefore \[ [T^{-1}(a^{*}),T^{-1}(b^{*})] = T^{-1}\big((-1)^{|a||b|} \mathfrak{ad}^{\star}(T^{-1}(b^{*}))(a^{*}) - \mathfrak{ad}^{\star}(T^{-1}(a^{*}))(b^{*}) \big). \] Thus $T^{-1}$ is an invertible anti-$\mathcal O$-operator of the Hom-Lie superalgebra $(\mathcal A,[\cdot,\cdot],\alpha)$ associated with the representation $(\mathcal A^{*},\mathfrak{ad}^{\star},(\alpha^{-1})^{*})$ which gives by Proposition \ref{inver-O-oper} that $T^{-1}$ is automatically strong super-anti-$\mathcal O$-operator on $(\mathcal A,[\cdot,\cdot],\alpha)$. By Corollary \ref{Hom-ant-pre-Lie-by-ant-R-B}, there exists a compatible Hom-anti-pre-Lie superalgebra structure $\circ$ on $(\mathcal A,[\cdot,\cdot],\alpha)$ defined by \begin{equation}\label{eq:Homproduct} x\circ y=-T^{-1}\big(\mathfrak{ad}^{\star}(x)T(y)\big),\qquad \forall x,y\in\mathcal H(\mathcal A). \end{equation} Finally, for any $x,y,z\in\mathcal H(\mathcal A)$, we obtain 
    \begin{eqnarray*}
     \mathfrak B(\alpha(x)\circ\alpha( y),\alpha^2(z)) &=&\langle T(\alpha(x)\circ\alpha( y)),\alpha^2(z)\rangle\\&=&-\langle \mathfrak{ad}^{\star}(\alpha(x))T(\alpha(y)),\alpha^2(z)\rangle\\&=&(-1)^{|x||y|}\langle T(\alpha(y)),[x,z]\rangle \\&=&(-1)^{|x||y|}\mathfrak B(\alpha(y),[x,z]). 
    \end{eqnarray*}
    Hence the Hom-anti-pre-Lie product $\circ$ satisfies \[ \mathfrak B(\alpha(x)\circ\alpha( y),\alpha^2(z))=(-1)^{|x||y|}\mathfrak B(\alpha(y),[x,z]), \qquad \forall x,y,z\in\mathcal H(\mathcal A). \] Therefore the conclusion follows. 
    \end{proof}
    \begin{cor}\label{ant-h-pr-Lie-from-mult-nondeg-bil}
  Let $\mathfrak{B}$ be a nondegenerate super-commutative $2$-cocycle on a multiplicative Hom-Lie superalgebra $(\mathcal A,[\cdot,\cdot],\alpha)$ satisfying Eq. \eqref{cond-mult-bil-form}. Then there exists a compatible Hom-anti-pre-Lie superalgebra structure $\circ$ on $(\mathcal A,[\cdot,\cdot],\alpha)$ given by \begin{equation}\label{mult-inv-Hom-anti-pre-Lie}\mathfrak{B}(x\circ y,\alpha(z))=(-1)^{|x||y|}\mathfrak{B}(\alpha(y),[x,z]), \;\;\forall x,y,z\in \mathcal H(\mathcal A). \end{equation} 
    Conversely, let $(\mathcal A,\circ,\alpha)$ be a Hom-anti-pre-Lie superalgebra and $(\mathcal A, [\cdot,\cdot],\alpha)$ be the sub-adjacent Hom-Lie
superalgebra. Then the natural nondegenerate symmetric bilinear form $\mathfrak B_d$ defined by Eq. \eqref{natur-biln-form-H-assoc-sup} is a super-commutative $2$-cocycle on the Hom-Lie superalgebra $\mathcal A\ltimes_{\mathfrak{ad}^\star}\mathcal A^\ast$.
        
    \end{cor}
\begin{df}
A bilinear form $\mathfrak B$ on a Hom-anti-pre-Lie superalgebra $(\mathcal A,\circ,\alpha)$ is called \textbf{Invariant} if Eq. \eqref{inv-Hom-anti-pre-Lie} holds.    
\end{df}
\begin{cor}\label{cor:sup-sym-bil-foris-2-coc}
Any even super-symmetric bilinear form on a multiplicative Hom-anti-pre-Lie superalgebra  $(\mathcal A,\circ,\alpha)$ is a super-commutative $2$-cocycle on the sub-adjacent Hom-Lie superalgebra $(\mathfrak g(\mathcal A),[\cdot,\cdot],\alpha)$. Conversely, a nondegenerate super-commutative $2$-cocycle on a multiplicative Hom-Lie superalgebra $(\mathcal A,\circ,\alpha)$ is invariant on the compatible Hom-anti-pre-Lie superalgebra given by Eq. \eqref{inv-Hom-anti-pre-Lie}.    
\end{cor}
\begin{proof}
Let $x,y,z\in\mathcal H(\mathcal A)$, then by Eq. \eqref{mult-inv-Hom-anti-pre-Lie}, we have
$$\mathfrak B([x,y],\alpha(z))=\mathfrak B(x\circ y,\alpha(z))-(-1)^{|x||y|}\mathfrak B(y\circ x,\alpha(z))=(-1)^{|x||y|}\mathfrak B(\alpha(y),[x,z])-\mathfrak B(\alpha(x),[y,z]).$$
Therefore $\mathfrak B$ is a super-commutative $2$-cocycle on the sub-adjacent Lie algebra $(\mathfrak g(\mathcal A),[\cdot,\cdot],\alpha)$. The second
half part follows from Corollary \ref{ant-h-pr-Lie-from-mult-nondeg-bil}.
\end{proof}

Recall that two representations $(V_1,\rho_1,\beta_1)$ and $(V_2,\rho_2,\beta_2)$ of
a Lie superalgebra $(\mathcal A,[\cdot,\cdot],\alpha)$ are called equivalent,
if there exists an even linear isomorphism

$\Psi:V_1\longrightarrow V_2$, such that
\begin{align*}
\Psi\circ\beta_1&=\beta_2\circ\Psi,\\
\Psi\bigl(\rho_1(x)v\bigr)&=
\rho_2(x)\Psi(v),
\qquad \forall\,x\in\mathcal H(\mathcal A),\; v\in\mathcal H(\mathcal V_1).
\end{align*}
\begin{prop}\label{Prop:nondeg-bilin-iif-sub-adj}
There exists an even nondegenerate invariant bilinear form on a multiplicative
Hom-anti-pre-Lie superalgebra $(\mathcal A,\circ,\alpha)$ if and only if its sub-adjacent
Hom-Lie superalgebra has two representations
$(\mathcal A,-\mathfrak L_{\circ},\alpha)$ and $(\mathcal A^{*},\mathfrak{ad}^{\star},(\alpha^{-1})^*)$,
and they are equivalent.
\end{prop}

\begin{proof}
Suppose that there is an even linear isomorphism
$\Psi:\mathcal A\rightarrow\mathcal A^{*}$, it satisfies
\[
\Psi(-\mathfrak L_{\circ}(x)y)
=\mathfrak{ad}^{\star}(\alpha(x))\Psi(\alpha(y)),
\qquad \forall\,x,y\in\mathcal H(\mathcal A).
\]
Define an even nondegenerate bilinear form $\mathfrak B$, such that
\begin{equation}\label{cond-isom-via-inv-bil-form}
\mathfrak B(x,y)=\langle \Psi(x),y\rangle,
\qquad \forall\,x,y\in\mathcal H(\mathcal A).
\end{equation}
Then we have
\begin{align*}
\mathfrak B(x\circ y,\alpha^2(z))
&=\langle \Psi(x\circ y),\alpha^2(z)\rangle \\
&=\langle \Psi(\mathfrak L_{\circ}(x)y),\alpha^2(z)\rangle \\
&=-\langle \mathfrak{ad}^{\star}(\alpha(x))\Psi(\alpha(y)),\alpha^2(z)\rangle \\
&\overset{(\ref{dual-adjoint-rep})}{=}(-1)^{|x||y|}
   \langle \Psi(\alpha(y)),[x,z]\rangle \\
&=(-1)^{|x||y|}\mathfrak B(\alpha(y),[x,z]).
\end{align*}
Thus, $\mathfrak B$ is invariant on $(\mathcal A,\circ,\alpha)$.

Conversely, suppose that $\mathfrak B$ is an even nondegenerate invariant
bilinear form on $(\mathcal A,\circ,\alpha)$.  Define a linear map
$\Psi:\mathcal A\rightarrow\mathcal A^{*}$ satisfying Eq. \eqref{cond-isom-via-inv-bil-form}, then it can be
inferred through similar proof as above that $\Psi$ is a linear
isomorphism, such that $(\mathcal A,-\mathfrak L_{\circ},\alpha)$ and
$(\mathcal A^{*}),\mathfrak{ad}^{\star},(\alpha^{-1})^*,$ are equivalent.
\end{proof}
\begin{cor}
Let $(\mathcal A,[\cdot,\cdot],\alpha)$ be a multiplicative Hom-Lie superalgebra. If there is a
nondegenerate super-commutative $2$-cocycle on
$(\mathcal A,[\cdot,\cdot],\alpha)$, then there exists a compatible
Hom-anti-pre-Lie superalgebra $\mathcal A,\circ,\alpha)$, which is given by
Eq. \eqref{inv-Hom-anti-pre-Lie}. Furthermore, $(\mathcal A,-\mathfrak L_{\circ},\alpha)$ and
$(\mathcal A^{*},\mathfrak{ad}^{\star},(\alpha^{-1})^*)$ are equivalent as representations of
$(\mathcal A,[\cdot,\cdot],\alpha)$.

Conversely, if there is a compatible Hom-anti-pre-Lie superalgebra
$(\mathcal A,-\mathfrak L_{\circ},\alpha)$ such that $(\mathcal A,-\mathfrak L_{\circ},\alpha)$ and
$(\mathcal A^{*},\mathfrak{ad}^{\star},(\alpha^{-1})^*)$ are equivalent as representations of
$(\mathcal A,[\cdot,\cdot],\alpha)$, then there exists an even
nondegenerate bilinear form $\mathfrak B$ on $(\mathcal A,[\cdot,\cdot],\alpha)$, and
it satisfies
\begin{equation}
\mathfrak B([x,y],\alpha(z))+\mathfrak B(\alpha(x),[y,z])
-(-1)^{|x||y|}\mathfrak B(\alpha(y),[x,z])=0.
\end{equation}
\end{cor}
\begin{proof}
It is straightforward to obtain these results by
Corollary \ref{cor:sup-sym-bil-foris-2-coc} and Proposition \ref{Prop:nondeg-bilin-iif-sub-adj}.
\end{proof}

\section{Hom-Poisson superalgebras with nondegenerate bilinear forms and compatible Hom-anti-pre-Poisson superalgebras}\label{Sec4}
In \cite{wi-oth}, the authors introduced the notion of noncommutative Hom-anti-pre-Poisson superalgebras and established a structural connection between noncommutative Hom-Poisson superalgebras and noncommutative Hom-anti-pre-Poisson superalgebras via the theory of strong anti-super-$\mathcal O$-operators. This construction extends the classical correspondence between Poisson-type algebras and their compatible anti-pre-algebra counterparts to the Hom-super setting and provides an effective approach for producing new examples of noncommutative Hom-anti-pre-Poisson superalgebras.

Motivated by these developments, we focus here on the commutative framework and investigate the role played by bilinear forms in the construction of compatible Hom-anti-pre-Poisson superalgebra structures. More precisely, we introduce the notion of a Hom-anti-pre-Poisson superalgebra and study its relationship with Hom-Poisson superalgebras through suitable cocycle conditions. In particular, we show that a nondegenerate super-symmetric bilinear form $\mathfrak B$ on a Hom-Poisson superalgebra $(\mathcal A,[\cdot,\cdot],\mu,\alpha)$ such that $\mathfrak B$ is a super-commutative $2$-cocycles on the Hom-Lie superalgebra $(\mathcal A,[\cdot,\cdot],\alpha)$ and a super-commutative Connes cocycles on the supercommutative Hom-associative superalgebra $(\mathcal A,\mu,\alpha)$ naturally induce compatible Hom-anti-pre-Poisson superalgebra structures. These results provide a unified framework for constructing Hom-anti-pre-Poisson superalgebras from Hom-Poisson superalgebras endowed with appropriate bilinear forms.

\begin{df}(\cite{O-N})
A \textbf{Hom-Poisson superalgebra} is a quadruplet $(\mathcal{A}, [\cdot, \cdot], \mu, \alpha)$ consisting of:
\begin{enumerate}
\item A Hom-Lie superalgebra $(\mathcal{A}, [\cdot, \cdot],\alpha)$,
\item A super-commutative Hom-associative superalgebra $(\mathcal{A},\mu, \alpha)$,
\item The Hom-Leibniz superalgebra identity
\begin{equation}\label{Homleibnizident}
[\alpha(x),\mu(y,z)]=\mu([x,y],\alpha(z))+(-1)^{|x||y|}\mu(\alpha(y),[x,z])
\end{equation}
is satisfied for all $x,y,z\in\mathcal{H}(\mathcal{A})$.
\end{enumerate} 
We say that a Hom-Poisson superalgebra
$(\mathcal{A},[\cdot,\cdot],\mu,\alpha)$ is \emph{multiplicative} if the
Hom-Lie superalgebra $(\mathcal{A},[\cdot,\cdot],\alpha)$ and the
Hom-associative superalgebra $(\mathcal{A},\mu,\alpha)$ are both
multiplicative.
\end{df}
\begin{exa}\label{exa-H-Poisson-sup}
Let $\mathcal A=\mathcal A_{\bar0}\oplus\mathcal A_{\bar1}$ be a $\mathbb Z_2$-graded vector space
where $\mathcal A_{\bar0}=\operatorname{span}\{e_1,e_2\}$ and $\mathcal A_{\bar1}=\operatorname{span}\{f\}$.

Define the even linear map $\alpha:\mathcal A\rightarrow\mathcal A$ by
\[
\alpha(e_1)=e_1,\qquad
\alpha(e_2)=2e_2,\qquad
\alpha(f)=2f.
\]
Define the bilinear multiplication $\mu:\mathcal A\times\mathcal A\to\mathcal A$ by
\[
\mu(e_1,e_1)=e_1,\qquad
\mu(e_1,e_2)=\mu(e_2,e_1)=2e_2,\qquad
\mu(e_1,f)=\mu(f,e_1)=2f,
\]
and all the remaining products are zero.

Define the bracket $[\cdot,\cdot]:\mathcal A\times\mathcal A\to\mathcal A$ by
\[
[e_1,e_2]=-[e_2,e_1]=2e_2,\qquad
[e_1,f]=-[f,e_1]=2f,
\]
and all the remaining brackets are zero.

Then $(\mathcal A,[\cdot,\cdot],\mu,\alpha)$ 
is a $3$-dimensional multiplicative Hom-Poisson superalgebra.    
\end{exa}
  
\begin{df}\label{repr-com-Hom-Poiss}
 A representation of a  Hom-Poisson superalgebra $(\mathcal A,[\cdot,\cdot],\mu,\alpha)$ is a quadruple $(V,\rho,\eta,\beta)$ such that $(V,\rho,\beta)$ is a representation of the Hom-Lie superalgebra $(\mathcal A,[\cdot,\cdot],\alpha)$ and $(V,\eta,\beta)$ is a representation of the super-commutative Hom-associative superalgebra $(\mathcal A,\mu,\alpha)$ satisfying for any $x,y\in\mathcal{H}(\mathcal A)$ the following equalities:
\begin{align}  
\eta([x,y])\beta&=\rho(\alpha(x))\eta(y)-(-1)^{|x||y|}\eta(\alpha(y))\rho(x),\label{cond-rep-com-Hom-Pois1}\\
\rho(\mu(x,y))\beta&=\eta(\alpha(x))\rho(y)+(-1)^{|x||y|}\eta(\alpha(y))\rho(x).\label{cond-rep-noncom-Hom-Pois2}
\end{align}
\end{df}
\begin{exa}
 Let $(\mathcal A,[\cdot,\cdot],\mu,\alpha)$ be a Hom-Poisson superalgebra. Then $(\mathcal A,\mathfrak{ad},\mathfrak L,\alpha)$ is a representation of $\mathcal A$, called the \textbf{adjoint representation} of $\mathcal A$.   
\end{exa}

Let $(\mathcal A,[\cdot,\cdot],\mu,\alpha)$ be a Hom-Poisson superalgebra. Then $(V,\rho,\eta,\beta)$ is a representation
of $\mathcal A$ if and only if the direct
sum $\mathcal A\oplus V$ of super-vector spaces is a (semi-direct product)
Hom-Poisson superalgebra by defining the multiplications on $\mathcal A\oplus V$
by Eqs. \eqref{Hom-ass-direct-sum1} and \eqref{Hom-Lie-direct-sum} respectively. We denote it by
$\mathcal A\ltimes_{\rho,\eta}V$.
\begin{prop}\label{Prop:dual-dep-H-Poiss-sup}
Let $(\mathcal A,[\cdot,\cdot],\mu,\alpha)$ be a Hom-Poisson superalgebra. If
$(V,\rho,\eta,\beta)$ is a representation of
$(\mathcal A,[\cdot,\cdot],\mu,\alpha)$, then
$(V^{*},\rho^{\star},-\eta^{\star},(\beta^{-1})^*)$
is also a representation of
$(\mathcal A,[\cdot,\cdot],\mu,\alpha)$.    
\end{prop}

Hence we get the following conclusion.
\begin{cor}
Let $(\mathcal A,[\cdot,\cdot],\mu,\alpha)$ be a Hom-Poisson superalgebra. Then
$(\mathcal A^{*},\mathfrak{ad}^{\star},-\mathfrak L_{\mu}^{\star},(\alpha^{-1})^*)$ is a representation of
$(\mathcal A,[\cdot,\cdot],\mu,\alpha)$, and the natural nondegenerate super-symmetric bilinear
form $\mathfrak B_{d}$ defined by Eq. \eqref{natur-biln-form-H-assoc-sup} on the resulting Hom-Poisson superalgebra
$\mathcal A\ltimes_{\mathfrak{ad}^{\star},\,-\mathfrak L_{\mu}^{\star}}\mathcal A^{*}$ is invariant on
both the super-commutative Hom-associative superalgebra
$\mathcal A\ltimes_{-\mathfrak L_{\mu}^{\star}}\mathcal A^{*}$ and the Hom-Lie superalgebra
$\mathcal A\ltimes_{\mathfrak{ad}^{\star}}\mathcal A^{*}$.    
\end{cor}

\begin{df}
 Let $(V,\rho,\eta,\beta)$ be a representation of a Hom-Poisson superalgebra $(\mathcal A,[\cdot,\cdot],\mu,\alpha)$. An even linear map $T:V\to\mathcal A$ is called an \textbf{anti-super-$\mathcal O$-operator} on $(\mathcal A,[\cdot,\cdot],\mu,\alpha)$ if it satisfying both \eqref{cond-ant-O-oper-ncomm-Hom-ass2} and \eqref{cond-ant-O-oper-Hom-Lie2}.
 
 An anti-super-$\mathcal O$-operator $T$ on a Hom-Poisson superalgebra $(\mathcal A,[\cdot,\cdot],\mu,\alpha)$ is called \textbf{Strong} if:
 \begin{enumerate}
\item $T$ is strong on the Hom-Lie superalgebra $(\mathcal A,[\cdot,\cdot],\alpha)$.
\item $T$ is strong on the super-commutative Hom-associative superalgebra $(\mathcal A,\mu,\alpha)$.
\item $T$ satisfying for any homogeneous elements $u,v,w$ of $V$ the following condition:
\begin{equation}\label{cond-strong-O-op-Hom-Pois-sup}
   \eta([T(u),T(v)])\beta(w)+(-1)^{|v||w|}\eta([T(u),T(w)])\beta(v)+(-1)^{|u|(|v|+|w|)}\rho(\mu(T(v),T(w)))\beta(u)=0. 
\end{equation}
 \end{enumerate}
 In particular, an anti-super-$\mathcal O$-operator $\mathcal R$ of $(\mathcal A,[\cdot,\cdot],\mu,\alpha)$ associated with the adjoint representation $(\mathcal A,\mathfrak {ad},\mathfrak L,\alpha)$ is called an \textbf{anti-Rota-Baxter operator (of weight zero)}, that is, $\mathcal R:\mathcal A\to\mathcal A$ is an even linear map satisfying conditions \eqref{ant-RB-oper-noncomm-ass} and \eqref{cond-Rota-Baxter-Oper}.
An anti-Rota-Baxter operator $\mathcal R$ is called \textbf{strong}, if it satisfies \eqref{strong-ant-RB-oper-noncomm-ass}, \eqref{cond-strong-Rota-Baxter-Oper} and the following condition.
\begin{equation}\label{cond-strong-R-B-Op-Hom-Poiss}
[[\mathcal R(x),\mathcal R(y)],\alpha(z)]+(-1)^{|y||z|}[[\mathcal R(x),\mathcal R(z)],\alpha(y)]+(-1)^{|x|(|y|+|
z|)}[\mu(\mathcal R(y),\mathcal R(z)),\alpha(x)]=0,   
\end{equation}
for all $x,y,z\in\mathcal H(\mathcal A)$.
\end{df}
\begin{df}
A \textbf{Hom-anti-pre-Poisson superalgebra} is a quadruple $(\mathcal A,\circ,\star,\alpha)$ consisting of a Hom-anti-pre-Lie superalgebra $(\mathcal A,\circ,\alpha)$ and a Hom-anti-Zinbiel superalgebra $(\mathcal A,\star,\alpha)$ such that the following conditions hold:
\begin{align}
(x\circ y-(-1)^{|x||y|}y\circ x)\star \alpha(z)
&=(-1)^{|x||y|}\alpha(y)\star(x\circ z)-\alpha(x)\circ(y\star z), \label{cond-H-ant-pre-Poiss1}\\
(x\star y+(-1)^{|x||y|}y\star x)\circ\alpha(z)
&=-\alpha(x)\star(y\circ z)-(-1)^{|x||y|}\alpha(y)\star(x\circ z), \label{cond-H-ant-pre-Poiss2}\\
(-1)^{|y||z|}\alpha(x)\star(z\circ y)+(-1)^{|x|(|y|+|z|)}\alpha(y)\star(z\circ x)
&=(-1)^{|z|(|x|+|y|)}\alpha(z)\circ(x\star y+(-1)^{|x||y|}y\star x)\nonumber\\&\quad-(x\star y+(-1)^{|x||y|}y\star x)\circ\alpha(z), \label{cond-H-ant-pre-Poiss3}
\end{align}
for all $x,y,z\in\mathcal H(\mathcal A)$.
\end{df}
\begin{exa}
Let $\mathcal A=\mathcal A_{\bar0}\oplus\mathcal A_{\bar1}$ be a $\mathbb Z_2$-graded vector space
where $\mathcal A_{\bar0}=\operatorname{span}\{e_1,e_2\}$ and $\mathcal A_{\bar1}=\operatorname{span}\{f\}$.

Define the even linear map $\alpha:\mathcal A\rightarrow\mathcal A$ by
$$
\alpha(e_1)=0,\qquad
\alpha(e_2)=e_1,\qquad
\alpha(f)=f.
$$
Define two even bilinear operations $\circ,\star:\mathcal A\times\mathcal A\to\mathcal A$ by
$$e_2\circ e_2=e_1,\;\;f\star f=e_1$$
and set all other products equal to zero.

Then $(\mathcal A,\circ,\star,\alpha)$ 
is a $3$-dimensional Hom-anti-pre-Poisson superalgebra.     
\end{exa}
\begin{thm}
Let $T:V\rightarrow \mathcal{A}$ be an anti-super-$\mathcal{O}$-operator on a Hom-Poisson-superalgebra $(\mathcal A,[\cdot,\cdot],\mu,\alpha)$ 
with respect to a representation $(V,\rho,\eta,\beta)$. Define the following binary operations $\circ_T,\star_T:V\times V\rightarrow V$ defined by
\begin{align*}
u\circ_T v&=-\rho(T(u))v, \\
u\star_T v&=-\eta(T(u))v. 
\end{align*}
Then $(V,\circ_T,\star_T,\beta)$ is a Hom-anti-Pre-Poisson superalgebra if and only if $T$ is strong.
\end{thm}
\begin{proof}
According to Theorems \ref{Hom-ant-Zinb-by-ant-O-oper} and \ref{Hom-ant-pre-Lie-by-ant-O-oper}, we conclude that  $(V,\circ_T,\beta)$ is a Hom-anti-pre-Lie superalgebra and $(V,\star_T,\beta)$ is a Hom-anti-Zinbiel superalgebra. It remains to show that Eqs. \eqref{cond-H-ant-pre-Poiss1}-\eqref{cond-H-ant-pre-Poiss3} are satisfied.\\
Let $u,v,w\in\mathcal H(V)$, we have:
\begin{align*}
(u\circ_T v-(-1)^{|u||v|}v\circ_T u)\star_T \beta(w)
&=\eta\Big(T\big(\rho(T(u))v-(-1)^{|u||v|}\rho(T(v))u\big)\Big)\beta(w)\\&\overset{(\ref{cond-ant-O-oper-Hom-Lie2})}{=}-\eta([T(u,T(v)])\beta(w)\\&\overset{(\ref{cond-rep-com-Hom-Pois1})}{=}-\rho(\alpha(T(u)))\eta((v))w+(-1)^{|u||v|}\eta(\alpha(T(v)))\rho(T(u))w\\&=-\rho(T(\beta(u)))\eta((v))w+(-1)^{|u||v|}\eta(T(\beta(v)))\rho(T(u))w\\&=-\beta(u)\circ_T(v\star_T w)+(-1)^{|u||v|}\beta(v)\star_T(u\circ_T z).   
\end{align*}
Therefore Eq. \eqref{cond-H-ant-pre-Poiss1} is satisfied. Similarly, by applying Eqs. \eqref{cond-ant-O-oper-ncomm-Hom-ass2} and \eqref{cond-rep-noncom-Hom-Pois2}, one can verify that Eq. \eqref{cond-H-ant-pre-Poiss2} is also satisfied.\\
For any homogeneous elements $u,v,w\in V$, we have:
\begin{align*}
LHS&=(-1)^{|v||w|}\beta(u)\star_T(w\circ_T v)+(-1)^{|u|(|v|+|w|)}\beta(v)\star_T(w\circ_T u)\\&\quad-(-1)^{|w|(|u|+|v|)}\beta(w)\circ_T(u\star_T v+(-1)^{|u||v|}v\star_T u)+(u\star_T v+(-1)^{|u||v|}v\star_T u)\circ_T\beta(w)\\&= (-1)^{|v||w|}\eta(\alpha(T(u)))\rho(T(w))v+(-1)^{|u|(|v|+|w|)} \eta(\alpha(T(v)))\rho(T(w))u\\&-(-1)^{|w|(|u|+|v|)}\rho(\alpha(T(w)))\eta(T(u))v- (-1)^{|w|(|u|+|v|)+|u||v|}\rho(\alpha(T(w)))\eta(T(v))u\\&+\rho\Big(T\big(\eta(T(u))v-(-1)^{|u||v|}\eta(T(v))u\big)\Big)\beta(w)\\&=-\Big((-1)^{|w|(|u|+|v|)}\rho(\alpha(T(w)))\eta(T(u))-(-1)^{|v||w|}\eta(\alpha(T(u)))\rho(T(w))\Big)v\\&-\Big((-1)^{|w|(|u|+|v|)+|u||v|}\rho(\alpha(T(w)))\eta(T(v))-(-1)^{|u|(|v|+|w|)}\eta(\alpha(T(v)))\rho(T(w))\Big)u\\&+\rho(\eta(T(u),T(v)))\beta(w)\\&=-(-1)^{|w|(|u|+|v|)}\eta([T(w),T(u)])\beta(v)-(-1)^{|u|(|v|+|w|)+|v||w|}\eta([T(w),T(v)])\beta(u)-\rho(\eta(T(u),T(v)))\beta(w). 
\end{align*}
Then $LHS=0$ if and only if condition \eqref{cond-strong-O-op-Hom-Pois-sup} holds. Equivalently, $T$ is strong, which completes the proof.
\end{proof}
\begin{cor}\label{H-ant-pre-Poiss-via-ant-R-B}
Let $\mathcal R:\mathcal A\rightarrow \mathcal{A}$ be an anti-Rota-Baxter operator of weight zero on a Hom-Poisson-superalgebra $(\mathcal A,[\cdot,\cdot],\mu,\alpha)$. Define the following binary operations $\circ_\mathcal{R},\star_\mathcal{R}:\mathcal{A}\times\mathcal{A}\rightarrow \mathcal{A}$ defined for any $x,y\in\mathcal{H}(\mathcal A)$ by
\begin{align*}
x\circ_\mathcal{R} y&=-[\mathcal R(x),y] \\
x\star_{R} y&=-\mu(\mathcal R(x),y).
\end{align*}
Then $(\mathcal A,\circ_\mathcal{R},\star_\mathcal{R},\alpha)$ is a Hom-anti-pre-Poisson superalgebra if and only if $\mathcal R$ is strong.    
\end{cor}
\begin{exa}
Let $(\mathcal A,[\cdot,\cdot],\mu,\alpha)$ be the Hom-Poisson superalgebra defined in Example~\ref{exa-H-Poisson-sup}. Define the even linear map $\mathcal R:\mathcal A\to\mathcal A$
by
\[
\mathcal R(e_1)=0,\quad
\mathcal R(e_2)=e_2,\quad \text{and}\quad
\mathcal R(f)=f.
\]
It is straightforward to verify that $\mathcal R$ is a strong anti-Rota-Baxter operator on $(\mathcal A,[\cdot,\cdot],\mu,\alpha)$. Consequently, Corollary~\ref{H-ant-pre-Poiss-via-ant-R-B} endows $\mathcal A$ with a Hom-anti-pre-Poisson superalgebra structure$
(\mathcal A,\circ_{\mathcal R},\star_{\mathcal R},\alpha)$,
where the operations $\circ_{\mathcal R}$ and $\star_{\mathcal R}$ are defined by
\[
x\circ_{\mathcal R}y=-[\mathcal R(x),y],\qquad
x\star_{\mathcal R}y=-\mu(\mathcal R(x),y),
\qquad \forall\,x,y\in\mathcal A.
\]
The nonzero products are given by
\[
e_2\circ_{\mathcal R}e_1
=-[\mathcal R(e_2),e_1]
=-[e_2,e_1]
=2e_2,\qquad
f\circ_{\mathcal R}e_1
=-[\mathcal R(f),e_1]
=-[f,e_1]
=2f,
\]
and
\[
e_2\star_{\mathcal R}e_1
=-\mu(\mathcal R(e_2),e_1)
=-\mu(e_2,e_1)
=-2e_2,\qquad
f\star_{\mathcal R}e_1
=-\mu(\mathcal R(f),e_1)
=-\mu(f,e_1)
=-2f,
\]
while all the remaining products vanish.

\end{exa}
\begin{prop}
Let $(\mathcal A,[\cdot,\cdot],\mu,\alpha)$ be a multiplicative Hom-Poisson superalgebra. Suppose that $\mathfrak B$ is a nondegenerate
super-symmetric bilinear form on $\mathcal A$ such that it is a super-commutative $2$-cocycle on $(\mathcal A,[\cdot,\cdot],\alpha)$ and a super-commutative Connes cocycle on
$(\mathcal A,\mu,\alpha)$.
Then there is a compatible Hom-anti-pre-Poisson superalgebra $(\mathcal A,\circ,\star,\alpha)$ in which $\star$ and $\circ$ are respectively defined by Eqs. \eqref{eq-sup-comm-con-cocy-to-H-ant-Zinb} and \eqref{inv-Hom-anti-pre-Lie}.

Conversely, let $(\mathcal A,\circ,\star,\alpha)$ be a Hom-anti-pre-Poisson superalgebra and the sub-adjacent
Hom-Poisson superalgebra be $(\mathcal A,[\cdot,\cdot],\mu,\alpha)$.
Then there is a Hom-Poisson superalgebra
$\mathcal A\ltimes_{-\mathfrak L_{\circ}^{\star},\mathfrak L_{\star}^{\star}}\mathcal A^{*}$,
and the natural nondegenerate symmetric bilinear form $\mathfrak B_d$
defined by Eq. \eqref{natur-biln-form-H-assoc-sup} is a super-commutative $2$-cocycle on the Hom-Lie superalgebra
$\mathcal A\ltimes_{-\mathfrak L_{\circ}^{\star}}\mathcal A^{*}$,   
and a supercommutative Connes cocycle on the
super-commutative Hom-associative superalgebra
$\mathcal A\ltimes_{\mathfrak L_{\star}^{\star}}\mathcal A^{*}$.
\end{prop}
\begin{proof}
By Theorems \ref{from-h-ass-to-H-anti-Zinb-via-comm-Connes-coc} and \ref{ant-h-pr-Lie-from-nondeg-bil}, it follows that there exist a Hom-anti-pre-Lie superalgebra $(\mathcal A,\circ,\alpha)$ compatible with the Hom-Lie superalgebra $(\mathcal A,[\cdot,\cdot],\alpha)$, and a Hom-anti-Zinbiel superalgebra $(\mathcal A,\star,\alpha)$ compatible with the Hom-associative superalgebra $(\mathcal A,\cdot,\alpha)$. The operations $\circ$ and $\star$ are given by \eqref{inv-Hom-anti-pre-Lie} and \eqref{eq-sup-comm-con-cocy-to-H-ant-Zinb}, respectively.\\
For any $x,y,z,t\in\mathcal H(\mathcal A)$, we have:
\begin{align*}
\mathfrak B((x\circ y-(-1)^{|x||y|}y\circ x)\star \alpha(z),\alpha(t))&=\mathfrak B((\alpha(x)\circ\alpha(y)-(-1)^{|x||y|}\alpha(y)\circ\alpha(x))\star \alpha^2(z),\alpha^2(t))
\\&= \mathfrak B([\alpha(x),\alpha(y)]\star \alpha^2(z),\alpha^2(t))\\&\overset{(\ref{eq-mult-sup-comm-con-cocy-to-H-ant-Zinb})}{=}-(-1)^{|z|(|x|+|y|)}\mathfrak B(\alpha^3(z),\mu([\alpha(x),\alpha(y)],\alpha(t)))\\&\overset{(\ref{Homleibnizident})}{=}-(-1)^{|z|(|x|+|y|)}\mathfrak B(\alpha^3(z),[\alpha^2(x),\mu(\alpha(y),t)])\\&+(-1)^{|x|(|y|+|z|)}\mathfrak B(\alpha^3(z),\mu(\alpha^2(y),[\alpha(x),t]))\\&\overset{(\ref{mult-inv-Hom-anti-pre-Lie}),(\ref{eq-mult-sup-comm-con-cocy-to-H-ant-Zinb})}{=}-(-1)^{|y||z|}\mathfrak B(\alpha^2(x)\circ\alpha^2(z),\mu(\alpha^2(y),\alpha(t)))\\&-(-1)^{|x|(|y|+|z|)}\mathfrak B(\alpha^2(y)\star\alpha^2(z),[\alpha^2(x),\alpha(t)])\\&\overset{(\ref{eq-mult-sup-comm-con-cocy-to-H-ant-Zinb}),(\ref{mult-inv-Hom-anti-pre-Lie})}{=}(-1)^{|x||y|}\mathfrak B(\alpha^2(y)\star(\alpha(x)\circ\alpha(z)),\alpha^2(t))\\&-\mathfrak B(\alpha^2(x)\circ(\alpha(y)\star\alpha(z)),\alpha^2(t))\\&=(-1)^{|x||y|}\mathfrak B(\alpha(y)\star(x\circ z),\alpha(t))\\&-\mathfrak B(\alpha(x)\circ(y\star z),\alpha(t))\\&=(-1)^{|x||y|}\mathfrak B(\alpha(y)\star(x\circ z)-\alpha(x)\circ(y\star z),\alpha(t)).
\end{align*}
By the fact that, $\alpha$ is bijective and $\mathfrak B$ is nondegenerate, we can conclude that
$$(x\circ y-(-1)^{|x||y|}y\circ x)\star \alpha(z)
=(-1)^{|x||y|}\alpha(y)\star(x\circ z)-\alpha(x)\circ(y\star z),\;\forall x,y,z\in\mathcal H(\mathcal A),$$
therefore Eq. \eqref{cond-H-ant-pre-Poiss1} is satisfied. Similarly, we can show that Eqs. \eqref{cond-H-ant-pre-Poiss1} and \eqref{cond-H-ant-pre-Poiss1} are satisfied.\\
Conversely, let
$(\mathcal A,\circ,\star,\alpha)$
be a Hom-anti-pre-Poisson superalgebra and let
$(\mathcal A,[\cdot,\cdot],\mu,\alpha)$
be its sub-adjacent Hom-Poisson superalgebra. Then $(\mathcal A,-\mathfrak L_\circ,-\mathfrak L_\star,\alpha)$ is a representation of $(\mathcal A,[\cdot,\cdot],\mu,\alpha)$. By Proposition \ref{Prop:dual-dep-H-Poiss-sup}, $(\mathcal A,-\mathfrak L_\circ^\star,\mathfrak L_\star^\star,(\alpha^{-1})^*)$ is also representation of $(\mathcal A,[\cdot,\cdot],\mu,\alpha)$. It is straightforward to show that $\mathfrak B_d$ is a super-commutative $2$-cocycle on the Hom-Lie superalgebra
$\mathcal A\ltimes_{-\mathfrak L_{\circ}^{\star}}\mathcal A^{*}$,   
and a supercommutative Connes cocycle on the
super-commutative Hom-associative superalgebra
$\mathcal A\ltimes_{\mathfrak L_{\star}^{\star}}\mathcal A^{*}$. This completes the proof.
\end{proof}


\end{document}